\documentclass[11 pt]{article}
\usepackage{amssymb,amsfonts,amsmath,latexsym,verbatim,amscd,amsthm,graphicx,epsfig,color,hyperref
}

\newcommand{\be}{\begin{equation}}
\newcommand{\ee}{\end{equation}}
\newcommand{\bes}{\begin{equation*}}
\newcommand{\ees}{\end{equation*}}

\newcommand{\rr}{{\mathbb R}}
\newcommand{\ms}{{\mathbb S}}

\newcommand{\hh}{{\mathbb H}}

 \newtheorem{teo}{Theorem}[section]
 
 \newtheorem{lemma}[teo]{Lemma}
 \newtheorem{prop}[teo]{Proposition}

 \newtheorem{rem}[teo]{Remark}
\numberwithin{equation}{section}

 \newcommand{\calL}{{\mathcal L}}

\newcommand{\lp}{\left(}
\newcommand{\rp}{\right)}

\newcommand{\divergence}{\text{div}}

\begin{document}

\title{\bf Some constructions for the fractional \\Laplacian on noncompact manifolds}

 \author{V. Banica,\footnote{Laboratoire Analyse et probabilit\'es (EA
2172), Universit\'e d'Evry, 23 Bd. de
France, 91037 Evry, France, valeria.banica@univ-evry.fr. Partially supported by the French ANR projects R.A.S.
ANR-08-JCJC-0124-01 and SchEq ANR-12-JS01-0005-01.}\\M.d.M. Gonz\'alez,\footnote{Universitat Polit\`ecnica de Catalunya, ETSEIB-MA1, Av. Diagonal 647, 08028 Barcelona, Spain, mar.gonzalez@upc.edu. Supported by grants MTM2011-27739-C04-01 (Spain), and 2009SGR345 (Catalunya).}\\
M. S\'aez,
\footnote{Pontificia Universidad Cat\'olica de Chile. Santiago
Avda. Vicu\~na Mackenna 4860, Macul, 6904441 Santiago, Chile, mariel@mat.puc.cl. Supported by
 grants  Fondecyt regular 1110048  and proyecto Anillo ACT-125, CAPDE, P. Universidad
 Cat\'olica de Chile}}

\maketitle

\abstract{We give a definition of the fractional Laplacian on some noncompact manifolds, through an extension problem introduced by Caffarelli-Silvestre. While this definition in the compact case is straightforward, in the noncompact setting one needs to have a precise control of the behavior of the metric at infinity and geometry plays a crucial role. First we give explicit calculations in the hyperbolic space, including a formula for the kernel and a trace Sobolev inequality. Then we consider more general noncompact manifolds, where the problem reduces to obtain suitable upper bounds for the heat kernel. }
\tableofcontents

\section{Introduction and statement of the results}

There are extensive works involving fractional order operators. In particular, nonlinear or free boundary problems involving fractional powers of the Laplacian $(-\Delta)^\gamma$ appear naturally in applications (see for instance \cite{Valdinoci:long-jump,StingaTorrea} and the references therein). As pseudodifferential operators, the classical definitions involve functional analysis and singular integrals. They are nonlocal objects, which means that a priori estimates and maximum principles are not easy to obtain. However, in the Euclidean case, Caffarelli and Silvestre have developed in \cite{CS} an equivalent definition using an extension problem in one more dimension in terms of a degenerate elliptic equation with $\mathcal A_2$ weight, of the type studied by Fabes-Jerison-Kenig-Serapioni \cite{FKS,FJK}.

On the other hand, from the geometry side there is the work of Graham-Zworski \cite{Graham-Zworski:scattering-matrix} that studies a general class of conformally covariant operators $P_\gamma$ defined on a compact manifold $M^n$. These operators are defined through scattering theory \cite{Mazzeo-Melrose:meromorphic-extension} when $M$ is  the boundary $M^n$ of a conformally compact Einstein manifold. In \cite{Chang-Gonzalez} both the geometrical and the PDE points of view were  reconciled and, in particular, the fractional Laplacian on the sphere $\ms^n$ (or $\mathbb R^n$ through stereographic projection) is defined from scattering theory in the Poincar\'e ball.

It is possible then to formulate fractional Yamabe-type problems for $P_\gamma$, as considered in \cite{Gonzalez-Qing}, where the main ingredients needed in the proof are a Hopf's maximum principle, elliptic estimates and a sharp Sobolev trace inequality. These are shown by means of the extension formulation of Caffarelli-Silvestre its generalization on manifolds by \cite{Chang-Gonzalez}.

However, if $M^n$ is a noncompact manifold with a Riemannian metric $g_M$, these methods are not available in general since it is not clear how to define fractional order operators in the noncompact setting. One can give a reasonable definition when $M$ is an open dense set in a compact manifold $\hat M$ and the metric $g_M$ is conformally related to a smooth metric $\hat g$ on $\hat M$. Namely, we can define $P_\gamma$ by demanding that a conformally covariant relationship holds. Note, however, that this is not as simple as it first appears. In \cite{Gonzalez-Mazzeo-Sire} singular fractional Yamabe problems were considered in the particular case that $M=S^n\backslash \Lambda$, where the singular set $\Lambda$ is a smooth $k$-dimensional submanifold and $g_M$ a complete metric  with controlled growth near the singular set. Not much is known in the general noncompact case.

In the present paper we try to formulate an extension problem for the fractional Laplacian $(-\Delta_M)^\gamma$ on hyperbolic space and on some other noncompact manifolds. It is a very interesting open question to set up a conformally covariant version of the operator (that has the same principal symbol as $(-\Delta_M)^\gamma$); we hope to return to this problem elsewhere.

More precisely, we will give sufficient conditions on the underlying manifold for the following to hold:

\begin{teo}\label{thm-extension}
Let $M^n$ be a $n$-dimensional complete, non-compact manifold with a Riemannian metric $g_M$ satisfying any of the conditions in Proposition \ref{prop:summary}, which in particular include hyperbolic space $\mathbb H^n$. Let $g$ be the product metric on $M\times \rr_+$ given by $g=g_{M}+dy^2$.
Set $\gamma\in(0,1)$ and $a=1-2\gamma$. For any given $f\in H^\gamma(M)$, there exists a unique
solution of the extension problem
\begin{equation}
\left\{\begin{split}
&\textrm{div}_g\,(y^a\,\nabla_g u)(x,y) = 0 \hbox{ for } (x,y)\in M\times \rr_+,\\
&u(x,0)=f(x) \qquad \hbox{ for } x\in M,
\end{split}\right.
\label{extension}\end{equation}
Moreover, the fractional Laplacian on $M$ is well defined and can be recovered through
\begin{equation}\label{Neumann}(-\Delta_M)^\gamma f=-d_{\gamma}\lim_{y\to 0}y^a \,\partial_y u,\end{equation}
for a constant
\begin{equation}
\label{constant-d}d_{\gamma}=2^{2\gamma-1}\frac{\Gamma(\gamma)}{\Gamma(1-\gamma)}.\end{equation}
\end{teo}
In view of the definition of the metric $g$, the extension problem \eqref{extension} writes
\begin{equation}\label{extproduct}
\left\{\begin{split}
&\partial_{yy} v + \frac{a}{y}\, \partial_y v +\Delta_{M} v \,(x,y) = 0 \hbox{ for } (x,y)\in M\times \rr_+,\\
&v(x,0)=f(x) \qquad \hbox{ for } x\in M.
\end{split}\right.
\end{equation}

In the case of hyperbolic space one is able to carry out very explicit calculations, since Fourier analysis and harmonic analysis tools are available. Indeed, $\mathbb H^n$ is the simplest example of a symmetric space of rank one. In addition, we give a precise formula for $(-\Delta_{\mathbb H^n})^\gamma$ in terms of a singular integral obtained as convolution with a well behaved kernel (see Theorem \ref{thm:singular-integral}). Elliptic a priori estimates may be obtained by understanding the asymptotics of this kernel; we show H\"older estimates as an application.

These results allow to set up semilinear problems for the fractional Laplacian on hyperbolic space. When the nonlinearity comes from a double well potential, one has existence and uniqueness of layer solutions, as well as some symmetry results (see \cite{Gonzalez-Saez-Sire}).

Semilinear equations for the usual Laplace-Beltrami operator on hyperbolic space were studied in \cite{Mancini-Sandeep, Castorina-Fabbri-Mancini-Sandeep}, for instance, in relation to conformal geometry. In their works,  a Sobolev inequality for the hyperbolic Laplacian appears naturally \cite{Mancini-Sandeep:Sobolev}. In \cite{Grunau-OuldAhmedou-Reichel}, the authors study the Paneitz operator on hyperbolic space, that is a conformally covariant operator of order 4. Higher order Sobolev inequalities were considered in \cite{Liu:Sobolev-hyperbolic}.

It is still an open question to set up the conformal geometry interpretation of the hyperbolic fractional Laplacian, and to study the associated fractional Yamabe problems.
One first step is to obtain trace Sobolev embeddings. This is the content of Theorem \ref{thm:trace-Sobolev}. The key idea, as in the Euclidean case, is to study the energy associated to problem \eqref{extproduct}, which allows to obtain sharp inequalities. Note that the best constant in this embedding is related to the Yamabe constant in the model case for the fractional Yamabe problem.

From the probability point of view, see the survey \cite{Cohen-Lifshits} for a construction of fractional L\'evy Brownian fields on hyperbolic space.

We note here that the fractional Laplacian on the torus constructed from the extension \eqref{extension} has been considered in \cite{Roncal-Stinga:torus}, using a double Fourier series expansion.  However, as we have mentioned, the question on noncompact manifolds, where Fourier analysis is not available, is more delicate and not much is known. Here we try to give a first approach.

The relation between heat kernel and fractional powers of an operator is a very old one.
From the spectral theory and functional calculus point of view, Stinga and Torrea \cite{StingaTorrea} show that one can define the fractional Laplacian on a domain  $\Omega\subset\mathbb R^n$ through the extension \eqref{extension} provided that the heat kernel associated to $-\Delta_\Omega$ exists and it satisfies some decay properties. Since the heat kernel on general noncompact manifolds has been extensively studied depending on the underlying geometry, we take this approach to prove Theorem  \ref{thm-extension}.

The organization of the paper is as follows: In section \ref{hyperbolic space}, based on the results in $\rr^n$, we concentrate on definitions and properties of the fractional Laplacian on hyperbolic space. We include a definition in terms of the Fourier transform (subsection \ref{hyperbolicviafourier}), an expression in terms of a singular integral (Theorem \ref{thm:singular-integral} and its proof in subsection \ref{poissonhyperbolic}) and  the relation with an appropriate extension problem (both in terms of a Poisson kernel and  and energy formulation, in subsections \ref{poissonhyperbolic} and \ref{energyhyperbolic} respectively). These results imply regularity, a Hopf's maximum principle  and sharp trace Sobolev inequality. We finish section \ref{hyperbolic space} by discussing how to extend the previous results to other harmonic groups (subsection \ref{harmonicgroups}). In section
\ref{noncompactmanifolds} we start by discussing  a general framework under which the results in \cite{StingaTorrea}  can be generalized to non-compact manifolds.
 In order to carry the construction of a fractional Laplacian, it is necessary to obtain bounds on the heat kernel. These bounds are discussed
in subsection \ref{heatestimates}. In the following subsections
we workout examples of manifolds that fulfill the required conditions. Among these are
symmetric spaces, some geometrically finite hyperbolic manifolds, a class of rotationally symmetric manifolds and certain manifolds with ends.

\textbf{Aknowledgements:} the authors would like to thank Thierry Couhlon for useful remarks on \S 3.2 and to the anonymous referee.

\section{The extension problem on hyperbolic space}\label{hyperbolic space}

The real hyperbolic space $\mathbb H^n$, $n\geq 2$, is the simplest example of Riemannian symmetric spaces of the noncompact type.
Fourier analysis on (noncompact) Riemannian symmetric spaces has been well studied. We refer to Helgason's books \cite{H,Helgason:libro2,Helgason:libro3} for the basic Fourier theory, to \cite{Anker:multipliers} for the theory of $L^p$ multipliers and to \cite{Anker-Ji} for heat kernel and Green function estimates.

\subsection{Notations and definitions}

Several models for the $n$-dimensional hyperbolic space $\hh^n$ have been considered in the literature. Here we will define it as the upper branch of a hyperboloid in $\rr^{n+1}$ with the metric induced by the Lorentzian metric in $\rr^{n+1}$ given by $-dx_0^2+dx_1^2+\ldots+dx_n^2$. More precisely, we take
\begin{eqnarray*}
\hh^n & =\{(x_0,\ldots,x_n)\in \rr^{n+1}: x_0^2-x_1^2-\ldots-x_n^2=1, \; x_0>0\} \\
&= \{x\in \rr^{n+1}: x= (\cosh r, \sinh r \,\omega), \; r\geq 0, \; \omega \in \ms^{n-1}\},
\end{eqnarray*}
with the metric
 $$g_{\mathbb H^n}=dr^2+\sinh^2 r\, d\omega^2,$$ where $d\omega^2$ is the metric on $\ms^{n-1}$.
Under these definitions the Laplace-Beltrami operator is given by
$$\Delta_{\hh^n}=\partial_{rr}+(n-1) \frac{\cosh r}{\sinh r}\, \partial_r+\frac{1}{\sinh^2 r}\,\Delta_{\ms^{n-1}}$$
and the volume element is $$\sinh^{n-1}r\;dr \, d\omega.$$
We denote by $[\cdot, \cdot]$ the internal product induced by the Lorentzian metric
$$[x,x']=x_0x_0'-x_1x_1'-\ldots -x_nx_n'.$$
The hyperbolic space is invariant under $SO(1,n)$, the group of Lorentz transformations of $\mathbb R^{n+1}$ that preserve this inner product. $\mathbb H^n$ can be actually defied as the quotient between the orbit $SO(1,n)e_0$ of the origin $e_0=(1,0,...,0)$ by the stabilizer of $e_0$, so that
$$\mathbb H^n\approx \frac{SO(1,n)}{SO(n)}.$$
In particular, using Cartan's decomposition, hyperboloids are symmetric spaces of rank one.
Finally, let us recall that, by means of stereographic projection through the hyperboloid origin, Poincar\'e's disk model is recovered, and from the disk model one obtains the model of the upper half-space by performing an inversion in a boundary point of the ball.

\subsection{Fourier transform and the fractional Laplacian}\label{hyperbolicviafourier}

We start by reviewing some basic facts about  Fourier transform on hyperbolic space, see \cite{GGG}, \cite{H} and \cite{Terras:book}.

Recall first that in $\rr^n$ the Fourier transform is given by
$$\hat{f}(\xi)=\int_{\rr^n} f(x)e^{- i x\cdot \xi} dx.$$
Notice that the functions  $e^{- i x\cdot \xi}$ are generalized (in the sense that they do not belong to $L^2$) eigenfunctions of the Laplacian associated to the eigenvalue $- |\xi|^2$. Moreover, the following inversion formula holds
$$f(x)=\frac{1}{(2\pi)^n}\int_{\rr^n} \hat{f}(\xi)e^{i x\cdot \xi} d\xi.$$

Similarly, in $\hh^n$ we consider the generalized eigenfunctions of the Laplace Beltrami operator:
$$h_{\lambda,\theta}(x)=[x,(1,\theta)]^{i\lambda-\frac{n-1}{2}}, \quad x\in \hh^n,$$
where $\lambda\in \rr$ and $\theta \in \ms^{n-1}$, that satisfy
$$\Delta_{\hh^n} h_{\lambda,\theta}=-\left(\lambda^2+\tfrac{(n-1)^2}{4}\right)h_{\lambda,\theta}.$$
In analogy with the Euclidean setting, the Fourier transform can be defined as (see \cite[Chapter III,  equation (4)]{H} or  \cite[Chapter 3, equation 3.5]{GGG}),
$$\hat{f}(\lambda, \theta)=\int_{\hh^n} f(x)\,h_{\lambda,\theta}(x)\,dx,$$
for $\lambda\in\rr$, $\omega\in \ms^{n-1}$, where $dx$ is the volume element in hyperbolic space. Moreover, the following inversion formula holds:
$$f(x)=\int_{-\infty}^{\infty}\int_{\ms^{n-1}}\bar{h}_{\lambda,\theta}(x)\hat{f}(\lambda,\theta)
\,\frac{d\theta \, d\lambda }{|c(\lambda)|^2},$$
where $c(\lambda)$ is the Harish-Chandra coefficient:
$$\frac{1}{|c(\lambda)|^2}=\frac 12\frac{|\Gamma(\frac{n-1}{2})|^2}{|\Gamma(n-1)|^2}\frac{|\Gamma(i\lambda+(\frac{n-1}{2})|^2}{|\Gamma(i\lambda)|^2}.$$
Similarly, Plancherel formula holds:

\begin{equation}\label{Plancherel}
\int_{\hh^n}|f(x)|^2\,dx=\int_{\rr\times \mathbb S^{n-1}}|\hat{f}(\lambda,\theta)|^2\frac{d\theta \; d\lambda }{|c(\lambda)|^2}.\end{equation}

It is easy to check by integration by parts for compactly supported functions (and consequently, for  every $f\in L^2(\hh^n)$) that
\begin{align*}\widehat{\Delta_{\hh^n} f}(\lambda, \theta)=&\int_{\mathbb{H}^n}\Delta_{\hh^n} f(x)h_{\lambda,\theta}(x)\,dx= \int_{\mathbb{H}^n}f(x)\Delta_{\hh^n}h_{\lambda,\theta}(x)\,dx\\
=&
-\left(\lambda^2+\tfrac{(n-1)^2}{4}\right)\hat{f}(\lambda, \theta).\end{align*}
Having in mind the theory of spherically symmetric multipliers, we define the fractional Laplacian on hyperbolic space $(-\Delta_{\mathbb{H}^n})^\gamma f$ as the operator that satisfies
$$\widehat{(-\Delta_{\mathbb{H}^n})^\gamma f}= \left(\lambda^2+\tfrac{(n-1)^2}{4}\right)^\gamma\hat{f}.$$
Due to the inversion formula we could write
\begin{equation*}\begin{split}
(-\Delta_{\mathbb{H}^n})^\gamma f (x)=&
\int_{-\infty}^{\infty}\int_{\ms^{n-1}}\bar{h}_{\lambda,\theta}(x)\left(\lambda^2+\tfrac{(n-1)^2}{4}
\right)^\gamma \hat{f}(\lambda,\theta)
\,\frac{d\theta \, d\lambda }{|c(\lambda)|^2}\\
=&\int_{-\infty}^{\infty}\left[\int_{\hh^n}\left(\lambda^2+\tfrac{(n-1)^2}{4}\right)^\gamma k_\lambda(x,x') \,f(x')\,
dx'\right] d\lambda,
\end{split}\end{equation*}
with
$$k_\lambda(x, x') =\frac{1}{|c(\lambda)|^2}\int_{\ms^{n-1}} \bar{h}_{\lambda,\theta}(x) \,h_{\lambda,\theta}(x')\,d\theta.$$
Since the Laplacian commutes with the action of $g\in SO(1,n)$ we have that
\begin{equation}\label{L}k_\lambda(x, x') =k_\lambda(gx, gx'),\end{equation}
and in particular,
$$k_\lambda(x, x') =k_\lambda(d_{\mathbb H^n}(x, x')).$$
Denote $\rho=d_{\mathbb H^n}(x,x')$. It is natural to define the singular kernel
\begin{equation}\label{kernel}\mathcal K_\gamma(\rho)=\int_{-\infty}^{\infty}\left(\lambda^2+\tfrac{(n-1)^2}{4}\right)^\gamma
k_\lambda(\rho)\,d\lambda,\end{equation}
and to (formally) calculate
$$(-\Delta_{\mathbb H^n})^\gamma f(x)=\int_{\mathbb H^n} f(x')\mathcal K_\gamma(d_{\mathbb H^n}(x,x'))\;dx' $$
Note that, by the invariance of the problem by $SO(1,n)$, we may rewrite the singular integral as follows: let $g_x$ be a transformation that takes $0$ into $x$, and change variables $x'=g_x \tilde x$. Since $g_x$ is an isometry, $\rho=d_{\mathbb H^n} (x',x)=d_{\mathbb H^n}(\tilde x,0)$. Then
\begin{equation}\label{definition0}(-\Delta_{\mathbb H^n})^\gamma f(x)= \int_{\hh^n}f(g_x\tilde x) \mathcal K_\gamma(d_{\mathbb H^n}(\tilde x,0)) \,d\tilde x.\end{equation}
In the following we aim to make this formulation precise. For that, we need to recall the following formulas for $k_\lambda$ (see, for instance, \cite{Banica}), where by $\frac{\partial_\rho}{\sinh\rho}$ we mean first the derivative operator $\partial_\rho$, and then the multiplication by $\frac {1}{\sinh\rho}$.

\begin{lemma} We have that
\begin{equation} k_\lambda(\rho)=
\left(\frac{\partial_\rho}{\sinh\rho}\right)^\frac{n-1}{2}(\cos\lambda\rho)\label{Kodd}\end{equation}
for $n\geq 3$ odd, and for $n\geq 2$ even,
\begin{equation}\label{Keven}
k_\lambda(\rho)=
\int_\rho^\infty\frac{\sinh r}{\sqrt{\cosh r-\cosh\rho}}
\left(\frac{\partial_r}{\sinh r}\right)^\frac{n}{2}(\cos\lambda r)\, dr.
\end{equation}
\end{lemma}

\begin{rem}
Note that $k_\gamma(\rho)$ is an even function, and so is $\mathcal K_\gamma(\rho)$ from \eqref{kernel}.
\end{rem}

We will also need this well known result (see \cite{Abramowitz-Stegun}):
\begin{lemma}\label{lemma-Bessel}
The solution of the ODE
\begin{equation*}
\label{Bessel1}\partial_{ss} \varphi+ \frac{\alpha}{s}\, \partial_s \varphi -\varphi = 0. \end{equation*}
may be written as $\varphi(s)=s^\nu \psi(s)$, for $\alpha=1-2\nu$, where $\psi$ solves the is the well known Bessel equation
\begin{equation}\label{Bessel2}
s^2\psi''+s\psi'-(s^2+\nu^2)\psi=0.
\end{equation}
In addition, \eqref{Bessel2} has two linearly independent solutions, $I_\nu,K_\nu$, which are the modified Bessel functions; their asymptotic behavior is given precisely by
\begin{equation*}\begin{split}
I_\nu(s)&\sim \frac{1}{\Gamma(\nu+1)}\lp\frac{s}{2}\rp^\nu\lp 1+\frac{s^2}{4(\nu+1)}+\frac{s^4}{32(\nu+1)(\nu+2)}+\ldots\rp,\\
K_\nu(s)&\sim \frac{\Gamma(\nu)}{2}\lp\frac{2}{s}\rp^{\nu}
\lp 1+\frac{s^2}{4(1-\nu)}+\frac{s^4}{32(1-\nu)(2-\nu)}+\ldots\rp
\\&\quad+\frac{\Gamma(-\nu)}{2}\lp\frac{s}{2}\rp^\nu\lp 1+\frac{s^2}{4(\nu+1)}+\frac{s^4}{32(\nu+1)(\nu+2)}+\ldots\rp,
\end{split}
\end{equation*}
for $s\to 0^+$, $\nu\not\in\mathbb Z$. And when $s\to +\infty$,
\begin{equation*}\begin{split}\label{asymptotic2}I_\nu(s)\sim \frac{1}{\sqrt{2\pi s}}e^s\lp1-\frac{4\nu^2-1}{8s}+\frac{(4\nu^2-1)(4\nu^2-9)}{2!(8s)^2}-\ldots \rp,\\
K_\nu(s)\sim \sqrt{\frac{\pi}{2s}}e^{-s}\lp1+\frac{4\nu^2-1}{8s}+\frac{(4\nu^2-1)(4\nu^2-9)}{2!(8s)^2}+\ldots \rp.\end{split}\end{equation*}
\end{lemma}

We now look at the asymptotics of the kernel defined in \eqref{kernel}:

\begin{teo} \label{thm-asymptotics}
For $\gamma\in(-1/2,1)$ it holds that
\begin{itemize}
\item For $n\geq 3$ odd,
\begin{equation}\label{K1}\mathcal K_\gamma(\rho)=\alpha_\gamma \left(\frac{\partial_\rho}{\sinh\rho}\right)^\frac{n-1}{2}
\rho^{-\frac{1}{2}-\gamma}K_{-\frac{1}{2}-\gamma}\left(\tfrac{n-1}{2} \rho\right),\end{equation}

\item When $n\geq 2$ is even,
$$\mathcal K_\gamma(\rho)=\alpha_\gamma
\int_\rho^\infty\frac{\sinh r}{\sqrt{\cosh r-\cosh\rho}}
\left(\frac{\partial_r}{\sinh r}\right)^\frac{n}{2}\left[r^{-\frac{1}{2}-\gamma}K_{-\frac{1}{2}-\gamma}\left(\tfrac{n-1}{2} r\right)\right]\, dr.$$
\end{itemize}
Here $K_{-\frac{1}{2}-\gamma}$ is the solution to the modified Bessel equation given by Lemma \ref{lemma-Bessel}.

Additionally, $\mathcal K_\gamma(\rho)$ has the asymptotic behavior:
\begin{itemize}
\item[\emph{i.}] As $\rho\to 0$,
\begin{equation}\label{asymptotics-zero}\mathcal K_\gamma(\rho)\sim \rho^{-n-2\gamma}.
\end{equation}
\item[\emph{ii.}] As $\rho\to\infty$,
$$\mathcal K_\gamma(\rho)\sim \rho^{-1-\gamma}e^{-(n-1)\rho}.$$
\end{itemize}
\end{teo}


\begin{proof}

We assume that $n$ is odd; the calculations for $n$ even are similar. Using \eqref{Kodd} we have formally for $n\geq 3$ that
\begin{align*}
\mathcal K_\gamma(\rho)=&\left(\frac{\partial_\rho}{\sinh\rho}\right)^\frac{n-1}{2}\left( \int_{-\infty}^{\infty}
\left(\lambda^2+\tfrac{(n-1)^2}{4}\right)^\gamma\cos\lambda\rho \,d\lambda\right)\\
=&\left(\frac{\partial_\rho}{\sinh\rho}\right)^\frac{n-1}{2}\left( \int_{-\infty}^{\infty}
\left(\lambda^2+\tfrac{(n-1)^2}{4}\right)^\gamma e^{-i\lambda\rho}\, d\lambda\right).
\end{align*}
The last equality follows from $\left(\lambda^2+\tfrac{(n-1)^2}{4}\right)^\gamma \sin \lambda\rho$ being odd.

Notice that the equality above holds when $\gamma<-\frac{1}{2}$, while for bigger values of $\gamma$ needs to be interpreted in the sense of distributions.
Hence, we need to compute first the distributional Fourier transform (in $\rr$) of $h(\lambda):= \left(\lambda^2+\tfrac{(n-1)^2}{4}\right)^\gamma$. Since $ \left(\lambda^2+\tfrac{(n-1)^2}{4}\right)^\gamma$ is a tempered distribution, it has Fourier transform which is also tempered distribution.

Moreover,
$$\left(\lambda^2+\tfrac{(n-1)^2}{4}\right) \partial_\lambda h=2\lambda\gamma \,h. $$
Taking Fourier transform, we have
$$\left(-\partial_{\rho\rho}+\tfrac{(n-1)^2}{4}\right) (i\rho \,\hat h)=2\gamma\,(i\partial_\rho \hat h), $$
or equivalently
$$ \rho\, \partial_{\rho\rho} \hat h+2(1+\gamma)\,\partial_\rho \hat h- \tfrac{(n-1)^2}{4}\rho\, \hat h=0.$$
By performing the change of variables $s=\frac{n-1}{2}\rho$ and denoting  $\varphi(s)=\hat h(\rho)$, we obtain the ODE
$$ \partial_{ss} \varphi+\frac{2(1+\gamma)}{s}\partial_s \varphi-  \varphi=0.$$
From Lemma \ref{lemma-Bessel} we have the solution may be written as
 $$\hat h(\rho)=\rho^{-\frac{1}{2}-\gamma}\left(\alpha_\gamma K_{-\frac{1}{2}-\gamma}\left(\tfrac{n-1}{2}\rho\right)+\beta_\gamma I_{-\frac{1}{2}-\gamma}\left(\tfrac{n-1}{2}\rho\right)\right),$$
where $K_{-\frac{1}{2}-\gamma}, I_{-\frac{1}{2}-\gamma},$ are the solutions to the  modified Bessel equation given in the same Lemma.

Since $h$ is a tempered distribution, $\hat{h}$ can at most have polynomial growth, hence, necessarily $\beta_\gamma=0$.
Notice that for $\gamma<0$ the function $\hat{h}$ is integrable near $\rho =0$, while for $\gamma>0$ it will need to be interpreted in the principal value sense
(see Theorem \ref{thm:singular-integral} below). Moreover,
recalling the asymptotic formulas for Bessel functions from Lemma \ref{lemma-Bessel} we have that
\begin{equation*}
\begin{split}
&\hat{h}\sim  \alpha_\gamma\,\rho^{-1-2\gamma}\,\lp\frac{(n-1)}{2}\rp^{-1-2\gamma}\,\Gamma(\frac{1}{2}+\gamma)2^{-\frac 12+\gamma}
\quad \mbox{as }\rho\to 0,\\
&\hat{h}\sim \alpha_\gamma \,\rho^{-\frac{1}{2}-\gamma} \sqrt{\frac{\pi}{(n-1)\rho}}e^{-\frac{n-1}{2}\rho} \quad \mbox{as }\rho\to\infty.
\end{split}\end{equation*}
Therefore \eqref{K1} and \eqref{asymptotics-zero} are proved.

\end{proof}

For $\gamma<0$, $\mathcal K_\gamma$ plays the role of the Riesz potential on Euclidean space. As a consequence, because of the asymptotics \eqref{asymptotics-zero}, we have that expression \eqref{definition0} for the fractional Laplacian is justified when $\gamma<0$. On the other hand, for $\gamma>0$ we have:

\begin{teo}\label{thm:singular-integral} Fix $\gamma\in(0,1)$. Let $f\in L^1(\mathbb H^n)\cap \mathcal C^{\alpha}(\mathbb H^n)$ for some $\alpha>2\gamma$. Then
\begin{align}\label{PVformula}(-\Delta_{\mathbb{H}^n})^\gamma f (x)= \hbox{ P.V. }  \int_{\hh^n}(f(x')-f(x)) \,\mathcal K_\gamma(d_{\mathbb H^n}(x',x)) \,dx'.
\end{align}
\end{teo}

The proof will be postponed until next section, where we shall show also that $\mathcal K_\gamma$ is a positive kernel for $\gamma\in(0,1)$.\\

There are several ways to define the Sobolev spaces on hyperbolic space and more generally on manifolds. We present some of them and refer to \cite{Triebel:II} and the references therein.
For a given  $n$-dim manifold $M$ with positive injectivity radius and bounded geometry the
Sobolev spaces $W^k_p(M)$ with $k$ integer were first defined as
$$W^k_p(M)=\{f\in L^p(M), \nabla_g^l f\in L^p(M),\forall 1\leq l\leq k\},$$
with  norm $\|f\|_{W^k_p(M)}=\Sigma_{l=0}^k\|\nabla_g ^l f\|_{L^p(M)}.$

The fractional spaces $H^\gamma_p(M)$ with $\gamma>0$ are
$$H^\gamma_p(M)=\{f\in L^p(M),\exists h\in L^p(M), f=(id-\Delta)^{-\gamma/2}h\},\hbox{ with norm } \|f\|_{H^\gamma_p(M)}=\|h\|_{L^p(M)}.$$
A similar definition is given also for $\gamma<0$. Lizorkin-Triebel spaces $F^\gamma_{pq}$, that rely on dyadic analysis, were defined on $M$ by the localization principle. Paley-Littlewood theorem makes the links between there three classes of spaces for $1<p<\infty$:
$$W^k_p(M)=H^k_p(M)=F^k_{p2}\, \mbox{ for } k\in\mathbb N,\quad H^\gamma_p(M)=F^\gamma_{p2}\, \mbox{ for } \gamma\in\mathbb R.$$
For $-\infty<\gamma_1\leq \gamma_2<+\infty$ and $0<p<+\infty$, the following Sobolev embedding holds
$$H^{\gamma_1}_{p}(M)\subseteq H^{\gamma_2}_{p}(M).$$

 Let $\gamma\in\mathbb R$ and $p\in(1,\infty)$. For stratified Lie groups or symmetric spaces (\cite{Folland},\cite{Anker:multipliers}, see also \S3 of \cite{Tataru:Strichartz-hyperbolic} for hyperbolic space), the following equivalence was proved:
$$H^\gamma_p(M)=\{f\in L^p(M), \|f\|_{L^p(M)}+\|(-\Delta)^{\frac \gamma 2}f\|_{L^p(M)}<\infty\}.$$
In our notation we will drop the subindex $p$ in the $p=2$ case.\\

Finally, the pointwise formula from Theorem \ref{thm:singular-integral} allows to prove some regularity estimates as the ones in \cite{Silvestre:regularity-obstacle}. The following proposition is obtained from \eqref{PVformula} together with the decay estimates from Theorem \ref{thm-asymptotics}. Since proof is the same one as in the Euclidean case, we refer the reader to \cite{Silvestre:regularity-obstacle}. The ingredients needed there are the estimates for the kernel and the Riesz transform (we recall that Riesz transform $R=\nabla \Delta^{-1/2}$ can be defined on $\mathbb H^n$, see \cite{HRiesz,BerensteinTarabusi}).

\begin{prop} Let $w=(-\Delta_{\mathbb H^n})^\gamma f$.
\begin{itemize}
\item[a.] Let $f\in \mathcal C^{k,\alpha}(\mathbb H^n)$, and suppose that $k+\alpha-2\gamma$ is not an integer. Then
$$w\in\mathcal C^{l,\beta},$$
where $l$ is the integer part of $k+\alpha-2\gamma$ and $\beta=k+\alpha-2\gamma-l$.

\item[b.] Assume that, for some $\alpha\in(0,1]$, $w\in\mathcal C^{0,\alpha}(\mathbb H^n)$ and $f\in L^\infty(\mathbb H^n)$. Then
\begin{itemize}
\item[b1)] If $\alpha+2\gamma\leq 1$, then $f\in \mathcal C^{0,\alpha+2\gamma}(\mathbb H^n)$. Moreover,
    $$\|f\|_{\mathcal C^{0,\alpha+2\gamma}}\leq C \lp \|f\|_{L^\infty}+\|w\|_{\mathcal C^{0,\alpha}}\rp.$$
\item[b2)] If $\alpha+2\gamma>1$, then $f\in \mathcal C^{1,\alpha+2\gamma-1}(\mathbb H^n)$. Moreover,
    $$\|f\|_{\mathcal C^{1,\alpha+2\gamma-1}}\leq C \lp \|f\|_{L^\infty}+\|w\|_{\mathcal C^{0,\alpha}}\rp.$$
\end{itemize}

\item[c.]Assume that $w,f\in L^\infty(\mathbb H^n)$. Then
\begin{itemize}
\item[c1)] If $2\gamma\leq 1$, then $f\in \mathcal C^{0,\alpha}(\mathbb H^n)$ for any $\alpha<2\gamma$. Moreover,
    $$\|f\|_{\mathcal C^{0,\alpha}}\leq C \lp \|f\|_{L^\infty}+\|w\|_{L^\infty}\rp.$$
\item[c2)] If $2\gamma>1$, then $f\in \mathcal C^{1,\alpha}(\mathbb H^n)$ for any $\alpha<2\gamma-1$. Moreover,
    $$\|f\|_{\mathcal C^{1,\alpha}}\leq C \lp \|f\|_{L^\infty}+\|w\|_{L^\infty}\rp.$$
\end{itemize}\end{itemize}
\end{prop}

\subsection{The extension problem}

For the rest of the section, we set $\gamma\in(0,1)$ and $a=1-2\gamma$. We use the ideas in \cite{CS} to reduce the extension problem \eqref{extension} to an ODE by taking Fourier transform in $x$.
Let $u:\hh^n\to \rr_+$ be a solution to
\begin{equation}\label{extension1}
\left\{\begin{split}
\partial_{yy} u + \frac{a}{y}\, \partial_y u +\Delta_{\mathbb H^n} u (x,y) = & 0 \qquad\hbox{ for } (x,y)\in \hh^n\times \rr_+,\\
u(x,0)=&f(x) \quad \hbox{ for } x\in \hh^n,
\end{split}\right.\end{equation}
where $g$ is the product metric on $\hh^n\times \rr_+$ given by $g=g_{\hh^n}+dy^2$.
Taking Fourier transform in the variable $x\in\hh^n$, one obtains
\begin{equation*}\left\{\begin{split}
\partial_{yy} \hat u + \frac{a}{y}\, \partial_y \hat u -\left(\lambda^2+\tfrac{(n-1)^2}{4}\right)\hat u &= 0,\\
 \hat u(\lambda,\theta,0)&=\hat f(\lambda,\theta),
\end{split}\right.\end{equation*}
that is an ODE for each fixed value of $\lambda,\theta$. With the change of variables \begin{equation}\label{changevar}
s=\left(\lambda^2+\tfrac{(n-1)^2}{4}\right)^{1/2} y,
\end{equation}
 $\varphi(s)=\hat u(\cdot, y)$ it gets transformed to
\begin{equation}\label{Bessel}\partial_{ss} \varphi+ \frac{a}{s}\, \partial_s \varphi -\varphi = 0, \end{equation}
that is described in Lemma \ref{lemma-Bessel}.
Let $\varphi_\gamma$ the unique solution of \eqref{Bessel} such that
$\varphi_\gamma(0)=1$, $\varphi_\gamma(\infty)=0$, which is explicitly written as
$\varphi_\gamma(s)=2^{1-\gamma}\Gamma(\gamma)^{-1}s^\gamma K_\gamma(s)$.
Then the solution to \eqref{extension1} is simply the inverse Fourier transform on $\mathbb H^n$ of
\begin{equation}\label{vFourier}\hat u(\lambda,\theta,y)=\hat f(\lambda,\theta)\,\varphi_\gamma\lp\left(\lambda^2+\tfrac{(n-1)^2}{4}\right)^{1/2}y\rp.\end{equation}
From the asymptotics at the origin, one may calculate
\be\label{constant1}\lim_{s\to 0}s^a\varphi'_\gamma(s)=-d_\gamma^{-1},\ee
for the constant given in \eqref{constant-d}, and we have that
\begin{equation}\label{formula10}\lim_{y\to 0}y^a \,\partial_y \hat{u}=\left(\lambda^2+\tfrac{(n-1)^2}{4}\right)^{\frac{1-a}{2}} \hat f\,\lim_{s\to 0}s^a\varphi_\gamma'(s)
=-\left(\lambda^2+\tfrac{(n-1)^2}{4}\right)^{\gamma} \hat f \,d_{\gamma}^{-1}.\end{equation}
Taking the inverse Fourier transform  in \eqref{formula10}, we complete the proof of Theorem \ref{thm-extension} in the case of hyperbolic space.\\

As we have mentioned, one of the advantages of the extension formulation \eqref{extension1} over the singular integral formulation from Theorem \ref{thm:singular-integral} is that it allows to prove elliptic type results. For instance, we show a weighted Hopf's maximum principle, which in the Euclidean case was considered  in \cite{Cabre-Sire:I}. We note here that the proof only depends on the structure of equation \eqref{extension1}:

\begin{prop} Let $p\in\mathbb H^n$ set $B_R(p)$ to be the ball in $\mathbb H^n$ centered at $p$ of radius $R$. Consider the cylinder $C_{R,1}=B_R(p)\times (0,1)\subset \mathbb H^n\times \mathbb R_+$. Let $u\in\mathcal (\overline{C_{R,1}})\cap H^1(C_{R,1},y^a)$ satisfy
\begin{equation*}
\left\{\begin{array}{ll}
\partial_{yy}u+\frac{a}{y}\partial_y u+\Delta_{\mathbb H^n} u\leq 0 &\quad \mbox{in }C_{R,1},\\
u>0&\quad \mbox{in }C_{R,1},\\
u(p,0)=0&
\end{array}\right.
\end{equation*}
Then $-\limsup_{y\to 0} y^a \frac{u(p,y)}{y}<0$. If, in addition, $y^a\partial_y u\in\mathcal (\overline{C_{R,1}})$, then
$$-\lim_{y\to 0}y^a\partial_y u(p,y)<0.$$
\end{prop}

\subsubsection{The Poisson kernel}\label{poissonhyperbolic}

In view of \eqref{vFourier}, we obtain that the solution of \eqref{extension} is
\begin{equation}\label{vformula}
\begin{split}
u(x,y)&=\int_{\hh^n} f(x')\int_{-\infty}^\infty k_{\lambda}(x,x')\,\varphi_\gamma\lp\left(\lambda^2+\tfrac{(n-1)^2}{4}\right)^{1/2}y\rp\,
d\lambda\,dx'\\
&=:\int_{\mathbb H^n}f(x') \mathcal  P^\gamma_y(d_{\mathbb H^n}(x,x'))\,dx'.
\end{split}\end{equation}
Therefore, from formula \eqref{L} the Poisson kernel may be written as
$$\mathcal P^\gamma_y(\rho)=\int_{-\infty}^\infty k_{\lambda}(\rho)\,\varphi_\gamma\lp\left(\lambda^2+\tfrac{(n-1)^2}{4}\right)^{1/2}y\rp\,d\lambda,$$
where $k_\lambda(\rho)$ is defined in \eqref{Kodd} and \eqref{Keven}.
For instance, for $n=3$, dimension for which the computations should be the simplest, the Poisson kernel is
$$\mathcal P^\gamma_y(\rho)=\frac{1}{\sinh\rho}\int_{-\infty}^\infty \,\varphi_\gamma
\lp\left(\lambda^2+\tfrac{(n-1)^2}{4}\right)^{1/2}y\rp\,\lambda\sin\lambda\rho\,\lambda\,d\lambda.$$

We may now conclude the proof of Theorem \ref{thm:singular-integral}. One has by using successively \eqref{formula10}, \eqref{vformula}, \eqref{extension1} and the fact that the Poisson kernel has integral equal to one (this is due to the fact that the extension solution for the constant function $f\equiv 1$ is the constant function $u\equiv 1$):
%
\begin{equation}\label{singform}\begin{split}
d_\gamma^{-1}(-\Delta_{\mathbb H^n})^\gamma f&=-\lim_{y\to 0} y^a\partial_y u =-\lim_{y\to 0} y^a\frac{u(x,y)-u(x,0)}{y}\\
&=-\int_{\mathbb H^n} \lim_{y\to 0} y^a\frac{\mathcal P_y^\gamma(d_{\mathbb H^n}(x,x'))}{y}\,(f(x')-f(x))\,dx'\\
&=-\int_{\mathbb H^n} \lim_{y\to 0} y^a\frac{\mathcal P_y^\gamma(d_{\mathbb H^n}(x,x'))-\mathcal P_0^\gamma(d_{\mathbb H^n}(x,x'))}{y}\,(f(x')-f(x))\,dx'\\
&=-\int_{\mathbb H^n} \lim_{y\to 0} y^a\partial_y\mathcal P_y^\gamma(d_{\mathbb H^n}(x,x'))\,(f(x')-f(x))dx'\\
&=-\int_{\mathbb H^n} \lim_{y\to 0} a_\gamma(y,d_{\mathbb H^n}(x,x'))\,(f(x')-f(x))\,dx',
\end{split}
\end{equation}
with
\begin{equation*}
\begin{split}a_\gamma(y,\rho)&=\int_{-\infty}^\infty k_\lambda (\rho)\left(\lambda^2+\tfrac{(n-1)^2}{4}\right)^{1/2} y^a\varphi'_\gamma\lp\left(\lambda^2+\tfrac{(n-1)^2}{4}\right)^{1/2}y\rp\,d\lambda\\
&=\int_{-\infty}^{\infty} k_\lambda(\rho)\left(\lambda^2+\tfrac{(n-1)^2}{4}\right)^{\gamma}s^a \varphi'_\gamma(s)\,d\lambda,
\end{split}\end{equation*}
for $s=\lp\lambda^2+\tfrac{(n-1)^2}{4}\rp^{1/2}y$. In view of the limit \eqref{constant1} we have
$$\lim_{y\to 0} a_\gamma(y,\rho)=-d_\gamma^{-1}\mathcal K_\gamma(\rho).$$
Now, since from Theorem \ref{thm-asymptotics} the kernel $\mathcal K_\gamma(\rho)$ behaves near zero as $\rho^{-n-2\gamma}$ and decays exponentially at infinity, in the last integral in \eqref{singform} we cannot replace the limit $\lim_{y\to 0} a_\gamma(y,\rho)$ by $-d_\gamma^{-1}\mathcal K_\gamma(\rho)$ because of the singular behavior at $\rho=0$. In order to make sense of the integration near zero the principal value has to be introduced, as it is done also in the Euclidean case (see for instance \cite{CS}). We obtain then
$$(-\Delta_{\mathbb H^n})^\gamma f=P.V.\int_{\mathbb H^n}(f(x')-f(x))\,\mathcal K_\gamma(d_{\mathbb H^n}(x,x'))\,dx',$$
where now the right hand side is well defined. Indeed, by doing a Taylor development for $x'$ near $x$ and using local charts, the first order term vanishes by parity, and the second order term behaves like $\rho^{\alpha-n-2\gamma}$ if $f\in\mathcal C^\alpha$, which is integrable at zero if $\alpha>2\gamma$. This completes the proof of Theorem \ref{thm:singular-integral}.

\begin{rem} \label{positivity}
The definition of the fractional Laplacian coming from the extension problem \eqref{extension1} also allows us to show that $(-\Delta_{\mathbb H^n})^\gamma$ is a positive operator. This is so because the first eigenvalue may be computed using this formulation. If $U_f(\cdot,y)=\mathcal P^\gamma_y * f$,
\begin{equation*}\begin{split}
\lambda_1&=\inf_f\frac{\int_{\mathbb H^n} f(-\Delta_{\mathbb H^n})^\gamma f}{\int_{\mathbb H^n} f^2} =\inf_f\frac{-\int_{\mathbb H^n} f d_\gamma\lim_{y\to 0}\lp y^a \partial_y U_f(\cdot,y)\rp}{\int_{\mathbb H^n} f^2} \\
&=
\inf_U\frac{\int_{\mathbb H^n\times \mathbb R_+} y^a|\nabla_{\mathbb H^n} U|^2 }{\int_{\mathbb H^n} (U|_{y=0})^2}\geq 0.
\end{split}\end{equation*}
This shows that the linear operator $(-\Delta_{\mathbb H^n})^\gamma$ is non-negative, so its kernel $\mathcal K_\gamma$ must be non-negative too.
\end{rem}

\subsubsection{Energy formulation}\label{energyhyperbolic}

We would like to re-write the weighted Dirichlet energy
\be\label{Dirichlet-energy}\|\nabla u\|^2_{L^2(\hh^n\times \rr_+,y^a)}=\int_{\hh^n\times \rr_+} y^a|\nabla_g u|^2 \, dx\, dy
.\ee
From \eqref{vFourier} we obtain
$$\widehat{\partial_y u}\,(\lambda,\theta,y)=\hat f(\lambda,\theta)\,\left(\lambda^2+\tfrac{(n-1)^2}{4}\right)^{1/2}\,\varphi_\gamma'\lp\left(\lambda^2+
\tfrac{(n-1)^2}{4}\right)^{1/2}y\rp.$$
Using Plancherel's formula,
\begin{equation*}\begin{split}
\int_{\hh^n}|\nabla_{\hh^n} f(x)|^2 \,dx&=\int_{\hh^n} f(x) (-\Delta_{\hh^n}\overline f(x))\,dx \\&
=\int_{-\infty}^\infty \int_{\mathbb S^{n-1}}\hat f(\lambda,\theta)\widehat{(-\Delta_{\hh^n})\overline f}(\lambda,\theta)\,\frac{d\theta\,d\lambda}{|c(\lambda)|^2}\\
&=\int_{-\infty}^\infty \int_{\mathbb S^{n-1}}|\hat f(\lambda,\theta)|^2\,\left(\lambda^2+\tfrac{(n-1)^2}{4}\right)\,
\frac{d\theta\,d\lambda}{|c(\lambda)|^2}.
\end{split}\end{equation*}
Substituting the previous expressions in \eqref{Dirichlet-energy} and using Plancherel's formula we infer that
\begin{equation*}
\begin{split}
\|\nabla u\|^2_{L^2(\hh^n\times \rr_+,y^a)}&=\int_{\hh^n\times \rr_+}  \left(|\partial_y u|^2+|\nabla_{\hh^n} u|^2\right) y^a\, dx\, dy\\
&
\begin{split}=\int_{-\infty}^\infty\int_{\mathbb S^{n-1}}\int_0^\infty \lp|\varphi_\gamma|^2+|\varphi_\gamma'|^2\rp\lp\left(\lambda^2+\tfrac{(n-1)^2}{4}\right)^{1/2}y\rp
\\ \qquad\cdot|\hat f(\lambda,\theta)|^2\,\left(\lambda^2+\tfrac{(n-1)^2}{4}\right)\,y^a\,dy\,
\frac{d\theta\,d\lambda}{|c(\lambda)|^2}.
\end{split}\end{split}\end{equation*}
By performing again the change of variables \eqref{changevar}, we obtain
\begin{equation*}\begin{split}
\|\nabla &u\|^2_{L^2(\hh^n\times \rr_+,y^a)}\\
&=\int_{-\infty}^\infty\int_{\mathbb S^{n-1}}|\hat f(\lambda,\theta)|^2\,\left(\lambda^2+\tfrac{(n-1)^2}{4}\right)^\gamma\int_0^\infty \lp|\varphi_\gamma(s)|^2+|\varphi_\gamma'(s)|^2\rp|\,s^a\,ds\,\frac{d\theta\,d\lambda}{|c(\lambda)|^2}\\
&=C_\gamma\int_{\hh^n}|(-\Delta_{\hh^n})^\gamma f(x)|^2\,dx,
\end{split}\end{equation*}
where
$$C_\gamma=I[\varphi_\gamma]:=\int_0^\infty \lp|\varphi_\gamma(s)|^2+|\varphi_\gamma'(s)|^2\rp|\,s^a\,ds,$$
is a positive constant that only depends on $\gamma$. Note that $\varphi_\gamma$ is the minimizer of the functional $I[\varphi]$.  In order to calculate the precise value of this constant, we multiply equation \eqref{Bessel} by $s^\gamma\varphi_\gamma$ and integrate between $\epsilon$ and $\infty$. Then
$$C_\gamma=\lim_{\epsilon\to 0}\int_\epsilon^\infty \lp|\varphi_\gamma(s)|^2+|\varphi_\gamma'(s)|^2\rp|\,s^a\,ds= -\lim_{\epsilon\to 0}\epsilon^a\varphi_\gamma(\epsilon)\varphi'_\gamma(\epsilon)=d_\gamma^{-1},$$
according to \eqref{constant1}.

After all this discussion, one may show, as in the real case (\cite{CS,Cotsiolis-Tavoularis:best-constants}), that:

\begin{teo}[Trace Sobolev embedding]\label{thm:trace-Sobolev}
For every $u\in W^{1,2}(\mathbb H^n\times\mathbb R_+, y^a)$, we have that
$$\|\nabla u\|^2_{L^2(\mathbb H^n\times\mathbb R_+, y^a)}\geq d_\gamma^{-1}\|u(\cdot,0)\|^2_{\dot H^\gamma(\mathbb H^n)}$$
for the constant given in \eqref{constant-d}, and with equality if and only if $u$ is the Poisson extension \eqref{vformula} of some function in $\dot H^\gamma(\mathbb H^n)$.
\end{teo}

\subsection{Other (noncompact) harmonic groups} \label{harmonicgroups}

The calculations in this section for the real hyperbolic space rely on the harmonic analysis available in this setting, by following the arguments done in \cite{CS} for the Euclidean case. In the same spirit, one can start to perform the same arguments in the case of more general harmonic groups as Damek-Ricci spaces, also known as harmonic NA groups. All symmetric spaces of rank one are included in this class. Part of the importance of this class is that it also contains non symmetric spaces, thus providing counterexamples to the Lichnerowicz conjecture (\cite{Damek-Ricci},\cite{Anker-Damek-Yacoub}, see also \cite{Anker-Pierfelice-Vallarino:Damek-Ricci}).

Instead of pursuing this method, we take the point of view of the next section using the heat kernel. Moreover, in Remark \ref{symmetric_kernel} we show that in several examples of this family, the fractional Laplacian can be represented as singular integral, where precise asymptotics can be given for the kernel.


\section{The fractional Laplacian on noncompact manifolds}\label{noncompactmanifolds}

The aim of this section is to construct the fractional Laplacian on a noncompact manifold $M$ through the extension problem \eqref{extension}, and to give sufficient conditions for the existence of a Poisson kernel. The main tool here is the study of the heat kernel on $M$; geometry plays a fundamental role.

\subsection{From heat to Poisson}\label{heattopoisson}

First we give some standard functional analysis background from \cite{Rudin:book,Yosida:book}. Let $L$ be a linear second order partial differential operator on $M$, that is assumed to be nonnegative, densely defined and self-adjoint in $L^2(M)$, for instance, $L=-\Delta_M$ for a complete manifold $M$. Then the spectral theorem can be applied to $L$, and consequently, there exists a unique resolution $E$ of the identity, supported on the spectrum of $L$ (which is a subset of $[0,\infty)$), such that
$$L=\int_0^\infty \lambda dE(\lambda).$$

Given a real measurable function $h$ on $[0,\infty)$, the operator $h(L)$ is formally defined as $h(L)=\int_0^\infty h(\lambda)dE(\lambda)$. The domain $\text{Dom}(h(L))$ of $h(L)$ is the set of functions $f\in L^2(\Omega)$ such that $\int_0^\infty |f(\lambda)|^2 dE_{h,h}(\lambda)<\infty$. In particular, one may define the heat diffusion semigroup generated by $L$ as $h(L)=e^{-tL}$, $t\geq 0$. Then:
\begin{enumerate}
\item For $f\in L^2(M)$, we have that $u=e^{-tL}f$ solves the evolution equation $u_t=-Lu$, for $t>0$.
\item $\|e^{-tL}f\|_{L^2(M)}\leq \|f\|_{L^2(M)}$, for all $t\geq 0$.
\item $e^{-tL}f\to f$ in $L^2(M)$ as $t\to 0^+$.
\end{enumerate}
One may also consider the fractional powers of $L$, given by $h(L)=L^\gamma$, $\gamma\in(0,1)$, with domain $\text{Dom}(L^\gamma)\supset \text{Dom}(L)$. Then:
\begin{enumerate}
\item When $f\in\text{Dom}(L^\gamma)$, we have $L^\gamma e^{-tL}f=e^{-tL}L^\gamma f$.
\item If $f\in\text{Dom}(L)$, then $\langle Lf,f\rangle=\|L^{1/2}f\|^2_{L^2(M)}$, where $\langle\cdot,\cdot,\rangle$ denotes the inner product in $L^2(M)$.
\item For $f\in\text{Dom}(L)$,
$$L^\gamma f(x)=\frac{1}{\Gamma(-\gamma)} \int_0^\infty (e^{-tL}f(x)-f(x))\frac{dt}{t^{1+\gamma}},\quad \text{in }L^2(M).$$
\end{enumerate}

In this framework, the paper \cite{StingaTorrea} relates the heat semigroup to the extension problem \eqref{extension} for the fractional Laplacian (their work is for domains in $\mathbb R ^n$ with a measure, but it is easily generalized when $M$ is a $n$-dimensional manifold). Let $f\in D(L^\gamma)$, and consider the extension problem
\begin{equation}\label{extension-ST}\left\{\begin{split}
\partial_{yy} u +\frac{a}{y}\partial_y u-L_x u &=0 \quad\text{in }M\times \mathbb R^+,\\
 u(\cdot,0)&=f \quad \text{on }M.
\end{split}\right.\end{equation}
They show that:

\begin{teo}\label{thm-ST1}
A solution to \eqref{extension-ST} is given by
$$u(x,y)=\frac{1}{\Gamma(\gamma)}\int_0^\infty e^{-tL}(L^\gamma f)(x)e^{-\frac{y^2}{4t}}\frac{dt}{t^{1-\gamma}},$$
and
$$\lim_{y\to 0^+} \frac{u(x,y)-u(x,0)}{y^{2\gamma}}= \frac{1}{2\gamma} \lim_{y\to 0^+} y^a \partial_y u(x,y)=\frac{\Gamma(-\gamma)}{2^{2\gamma}\Gamma(\gamma)}L^\gamma f(x).$$
Moreover, the following Poisson formula for $u$ holds
$$u(x,y)=\frac{y^{2\gamma}}{2^{2\gamma}\Gamma(\gamma)}\int_0^\infty e^{-tL} f(x)e^{-\frac{y^2}{4t}}\frac{dt}{t^{1+\gamma}}=\frac{1}{\Gamma(\gamma)}\int_0^\infty e^{-\frac{y^2}{4r}L}f(x)e^{-r}\frac{dr}{r^{1-\gamma}}=: P_y^\gamma f(x).$$
\end{teo}

These identities are to be understood in the $L^2$ sense. If addition, we make the following extra assumptions:
\renewcommand{\theenumi}{\Roman{enumi}}
\renewcommand{\labelenumi}{\theenumi}
\begin{enumerate}
\item \label{hypo1}The heat diffusion semigroup is given by integration against a nonnegative heat kernel $p_t(x,\zeta)$.
\item  \label{hypo2} The heat kernel belongs to the domain of $L$ and $\partial_t  p_t=L p_t$, where the $t$-derivative is understood in the classical sense.
\item \label{hypo3} \label{hyp2}Given $x$, there exists a constant $C_x$ and $\epsilon>0$ such that
\begin{equation}\label{assumption-ST}
\|p_t(x,\cdot)\|_{L^2(M)}+\|\partial_t p_t (x,\cdot)\|_{L^2(M)}\leq C_x (1+t^\epsilon)t^{-\epsilon}.
\end{equation}
\end{enumerate}
Then in the same paper \cite{StingaTorrea} the authors give a formula for the Poisson kernel for \eqref{extension-ST}:

\renewcommand{\theenumi}

\begin{teo}\label{thm-ST2}
Under the additional hypotheses \eqref{hypo1}, \eqref{hypo2} and \eqref{hypo3} we have:
\begin{itemize}
\item[1.] One may write $ P^\gamma_y f(x)=\int_{M} \mathcal P^\gamma_y(x,\zeta)f(\zeta)dv(\zeta)$, where the Poisson kernel is given by
    $$\mathcal P_y^\gamma(x,\zeta):=\frac{y^{2\gamma}}{2^{2\gamma}\Gamma(\gamma)}\int_0^\infty p_t(x,\zeta)e^{-\frac{y^2}{4t}}\frac{dt}{t^{1+\gamma}},$$
    and, for each $\zeta\in M$, is a $L^2$-function that verifies the first equation in \eqref{assumption-ST}.
\item[2.] ${\sup_{y\geq 0}} | P_y^\gamma f|\leq \sup_{t\geq 0}|e^{-tL} f| $ in $M$.
\item[3.] $\|P_y^\gamma f\|_{L^p(M)}\leq \|f\|_{L^p(M)}$, for all $y\geq 0$.
\item[4.] If $\lim_{t \to 0^+} e^{-tL} f=f$ in $L^p(M)$, then $\lim_{y\to 0^+}  P_y^\gamma f=f$ in $L^p(M)$.
\end{itemize}
\end{teo}

\subsection{Heat kernels on noncompact manifolds} \label{heatestimates}

Good references for heat kernel on manifolds are the book \cite{Davies:book} or the survey \cite{Grigoryan1998:survey}. Let $(M^n,g)$ be a Riemannian manifold with a metric $g$. Let $L=-\Delta_{M}$, the Laplace-Beltrami operator with respect to this metric. If $M$ is complete, then $L$ is self-adjoint on $\mathcal C_0^\infty(M)$.
 We have that $e^{-Lt}$ is a positivity-preserving one-parameter contraction semigroup on $L^p(M)$ for $1\leq p\leq \infty$. In particular, in $L^2(M)$ it has a strictly positive $\mathcal C^\infty$ kernel $p_t(x,y)$, $x,x'\in M$, $t>0$, satisfying:
\begin{enumerate}
\item[\emph{a.}] As a function of $t$ and $x$,
$$\partial_t p_t=\Delta_M p_t.$$
\item[\emph{b.}] When $t\to 0^+$,
$$p_t(\cdot,x')\to \delta_{x'}.$$
\item[\emph{c.}] The semigroup property $e^{(t+s)\Delta_M}=e^{t\Delta_M} e^{s\Delta_M}$, which reads as
\begin{equation}\label{semigroup}
p_{t+s}(x,x')=\int_M p_t(x,z) p_s (x,x')\,dz.
\end{equation}
\item[\emph{d.}] The symmetry
$$p_t(x,x')=p_t(x',x).$$
\end{enumerate}

It is clear then that the hypothesis of Theorem \ref{thm-ST1} are satisfied in this case. However, to pass from the heat kernel to the Poisson kernel as in Theorem \ref{thm-ST2} is a non-trivial issue in the case of non-compact manifolds since one needs to control the behavior at infinity.

We concentrate here into obtaining the bound \eqref{hypo3}. First, in order to get $L^2$ estimates for the kernel $p_t$ it is enough to have $L^\infty$ bounds. Indeed, for fixed $x\in M$,
$$\|p_t(x,\cdot)\|^2_{L^2(M)}\leq \sup_{x'\in M} \{p_t(x,x')\}\int_M p_t(x,x')\,dx'\leq \|p_t(x,\cdot)\|_{L^\infty(M)},$$
where we have used that, for all $t>0$ and $x\in M$, $\int_M p_t(x,x')dx'\leq 1,$
where $dx'$ is the volume element in the manifold $M$. For stochastically complete manifolds, this integral is exactly equal to one (see \cite{Grigoryan1999}).

Now, to estimate the time derivative of $p_t$, we recall that (see \cite{Grigoryan1993})
$$\int_{M}|\nabla^k p_t(x,x')|^2 \,dx'\leq \frac{1}{k} \int_{M}|p_{t/2^k}(x,x')|^2,\quad\mbox{for all }k\in\mathbb N.$$
We conclude that, in order to check hypothesis \eqref{hypo3} in Theorem \ref{thm-ST2}, it is enough to find good pointwise upper bounds for the heat kernel $p_t(x,x')$. This can be also seen by the holomorphy of the heat semigroup (see for instance \cite{Varopoulos}).

On the other hand, because of the semigroup property \eqref{semigroup} and the symmetry of the heat kernel,
\begin{equation}\label{formula50}
p_t(x,x)=\int_M p_{t/2}(x,z)^2\,dz.\end{equation}
Using the semigroup identity \eqref{semigroup} again and the Cauchy-Schwarz inequality,
$$p_t(x,x')=\int_M p_{t/2}(x,z) p_{t/2}(x',z)\,dz\leq
\lp \int_M p_{t/2}(x,z)^2\,dz\rp^{1/2}\lp \int_M p_{t/2}(x',z)^2\,dz\rp^{1/2},$$
which together with \eqref{formula50} it implies
$$p_t(x,x')\leq \sqrt{p_t(x,x) p_t(x',x')};$$
 i.e., if one knew good on-diagonal estimates for $p_t(x,x)$, this would imply a $L^\infty$ bound for the heat kernel $p_t(x,x')$. Also, we get $p_{2t}(x,x)\leq p_t(x,x)$ so if \eqref{hypo3} is satisfied up to some time $t_0$, then we get for all $t\geq t_0$ the upper bound $C_x(1+(t_0/2)^\epsilon)(t_0/2)^{-\epsilon}$. Therefore it is enough to verify \eqref{hypo3} for small times.

 Note that the heat kernel on Euclidean space has the explicit formula
 $$p_t(x,x')=\frac{1}{(4\pi t)^{n/2}}e^{-\frac{\rho^2}{4 t}},$$
 where $\rho=|x-x'|$.  The survey \cite{Grigoryan1998:survey} gives many examples for bounds of the heat kernel, in relation to Faber-Krahn and isoperimetric inequalities. For instance,
\begin{itemize}
\item Minimal submanifolds of $\mathbb R^N$.
\item Manifolds with sectional curvature bounded in between two constants (\cite{Cheng-Li-Yau}).
\item Cartan-Hadamard manifolds; which are geodesically complete simply connected non-compact Riemannian manifolds with non-positive sectional curvature (\cite{Cheeger-Yau,Debiard-Gaveau-Mazet}).
\end{itemize}
In all of these the heat kernel has the same asymptotic behavior as in the Euclidean case:
$$0\leq p_t(x,x')\leq \frac{C}{t^{n/2}}e^{-\frac{\rho^2}{ct}},$$
for some $c,C>0$, and where $\rho=d_M(x,x')$.



If $M$ is a complete manifold with non-negative Ricci tensor, heat kernel bounds were obtained by \cite{Li-Yau}. More generally, if $M$ is a complete manifold with lower bounded Ricci tensor,
 $$Ric\geq -(n-1)\beta^2$$
for some $\beta>0$, then, denoting $\lambda_1\geq 0$ the bottom of the spectrum of the operator $-\Delta_M$, we have the following heat kernel estimate \cite{Davies:upper-bounds},
\begin{equation}
\label{upper-heat2}
0\leq p_t(x,x')\leq c_\delta  |B(x,\sqrt t)|^{-1/2}|B(x',\sqrt t)|^{-1/2}e^{(\delta-\lambda_1)t}
e^{-\frac{\rho^2}{(4+\delta)t}}.
\end{equation}
We recall Bishop's comparison theorem, which states that if $0<R_1<R_2$, then
 \begin{equation}\label{comparison}
 \frac{|B(x,R_2)|}{|B(x,R_1)|}\leq \frac{|B_\beta(R_2)|}{|B_\beta(R_1)|},
 \end{equation}
where $B_\beta(R)$ is the volume of the geodesic ball of radius $R$ in the constant $-\beta^2$ sectional curvature space form, that may be calculated from the volume element
\begin{equation*}
J_\beta(\rho)=\left\{
\begin{split}
&\lp\tfrac{1}{\beta}\rp^{n-1}\sinh^{n-1} (\beta \rho),\quad \mbox{if}\quad \beta\neq 0,\\
&\rho^{n-1} ,\quad \mbox{if}\quad \beta= 0.
\end{split}\right.
\end{equation*}
In particular, for all $x\in M$,
$$\frac 1{|B(x,\sqrt{t})|}\leq \frac 1{|B(x,1)|}\frac{C}{t^n}.$$
%
It is clear that to prove the bound \eqref{hypo3} in the situations just discussed, one should ask for a lower bound for the volume of balls. More precisely, we seek $C>0$ such that
\begin{equation}\label{condition-balls}
|B(x,1)|\geq C>0 \quad \forall x\in M.
\end{equation}

Finally, for a manifold with doubling condition and local Poincar\'e inequality, we have as an upper-bound for the heat kernel \cite{Saloff-Coste:book}
$$p_t(x,x)\leq \frac{C}{|B(x,\sqrt{t})|}.$$
In view of the doubling condition, it is enough again to have \eqref{condition-balls} to get \eqref{hypo3}.

Summarizing the results of this section:

\begin{prop}\label{prop:summary}
Let $M$ be a complete noncompact Riemannian manifold. Hypothesis \eqref{hypo3} in Theorem \ref{thm-ST2} is fulfilled if $M$ satisfies any of these:
\begin{itemize}
\item $M$ is a Cartan-Hadamard manifold.
\item $Ric\geq -(n-1)\beta^2$ for some $\beta>0$
and \eqref{condition-balls} holds.
\item $M$ satisfies a volume doubling condition and the local Poincar\'e inequality.
\end{itemize}
\end{prop}

In the following, we look at several classes of admissible manifolds. The first two examples are classical and more or less explicit formulas for the heat kernel are known. Then we concentrate on spherically symmetric manifolds, that serve as a model for more general cases such as manifolds with ends.

\subsection{Admissible classes of manifolds}\label{admissible manifolds}

\subsubsection{Symmetric spaces}\label{symmetricspaces}
An explicit expression can be found for the heat kernel on hyperbolic space:
\begin{equation}p_t(\rho)= \frac{(2\pi)^{-\frac{d}{2}}}{2}\,t^{-\frac{1}{2}}e^{-\lp\frac{n-1}{2}\rp^2t}
\left(-\frac{\partial_\rho}{\sinh\rho}\right)^\frac{n-1}{2}e^{-\frac{\rho^2}{4t}}\label{Ptodd}\end{equation}
for $n\geq 3$ odd, and for $n\geq 2$ even,
\begin{equation}\label{Pteven}
p_t(\rho)= (2\pi)^{-\frac{d+1}{2}}\,t^{-\frac{1}{2}}e^{-\lp\frac{n-1}{2}\rp^2t}
\int_\rho^\infty\frac{\sinh s}{\sqrt{\cosh s-\cosh \rho}}
\left(\frac{\partial_s}{\sinh s}\right)^\frac{n}{2} e^{-\frac{s^2}{4t}}\, ds,
\end{equation}
where $\rho=d_{\mathbb H^n}(x,x')$. Moreover, global bounds were derived in \cite{Davies-Mandouvalos:kleinian}
\begin{equation}\label{approxhyperheat}
p_t(\rho)\approx \frac{(1+\rho)(1+\rho+t)^{\frac{n-3}{2}}}{t^\frac n2}e^{-\frac{(n-1)^2}{4}t-\frac{n-1}{2}\rho-\frac{\rho^2}{4t}},
\end{equation}
for all $\rho\geq 0$, $t>0$, in the sense that the ratio of the right hand side and the left hand side is uniformly bounded between two positive constants. These estimates have a generalisation to all symmetric spaces of rank one, and more generally to Damek-Ricci spaces, since in this case there is also an explicit formula for the heat kernel.

Global bounds on the heat kernel in the higher rank case were proved in \cite{Anker-Ostellari}. We recall that Ricci curvature is  bounded from below since is an Einstein manifold, that is  $Ric=-(\frac {m_2}4 +m_1)g$ (where as before $g$ is the Riemannian metric on $M$).
Here $m_1$ and $m_2$ stand for the two dimensions entering the algebraic construction of the space. For instance $m_1=0,m_2=n-1$ for the real hyperbolic space and  $m_1=1,m_2=2(n-1)$ for the complex hyperbolic space. Also, the volume element is of type $\sinh r^{m_1+m_2}\cosh r^{m_1}$. Therefore Theorem \ref{thm-ST2} applies by using directly the explicit sharp $L^\infty$ bounds on the heat kernel, or by using \eqref{condition-balls} and explicit formulas for the volume of the balls.
Moreover, we can represent the fractional Laplacian in the spirit of Theorem \ref{thm:singular-integral}.
\begin{rem}\label{symmetric_kernel}
From \cite{Anker-Ostellari} we have that
$$\mathcal P_y^\gamma(x,\zeta)\approx \int_0^\infty   \frac{\prod_{\alpha\in \Sigma_0^+}(1+\langle \alpha, H\rangle)(1+\langle \alpha, H\rangle+t)^{\frac{m_\alpha+m_{2\alpha}}{2}-1}}{t^{\frac n2+1+\gamma}}e^{-|\varrho|^2t-\langle \varrho, H\rangle-\frac{|H|^2}{4t}}dt.$$
The quantities $\varrho$, $H$, $\alpha$ and $\Sigma_0$ are related to the Cartan decomposition as follows:
If the associated Lie algebra $\mathfrak{g}$ decomposes as
$$\mathfrak{g}=\mathfrak{a} \oplus \mathfrak{m}\oplus \{\oplus_{\alpha\in\Sigma} \mathfrak{g}_\alpha\},$$ and
$\Sigma^+, \; \Sigma_0^+$ denote the set of positive and positive indivisible roots, then
\begin{equation*}\begin{split}
\varrho=&\frac{1}{2}\sum_{\alpha_\in \Sigma^+} m_\alpha \alpha, \hbox{ where }m_\alpha=\textrm{dim } \mathfrak{g}_\alpha,\\
G=& K(\textrm{exp } \overline{\mathfrak{a}^+})K, \\
\zeta^{-1}x&=k_1e^H k_2, \hbox{ with }k_1, k_2\in K, \; H\in \overline{\mathfrak{a}^+}.
\end{split}
\end{equation*}
For precise definitions of these quantities we refer the reader to \cite{Anker-Ostellari} and \cite{Helgason:libro2}.
\end{rem}

In a similar fashion as the proof of Theorem \ref{thm:singular-integral}, one can find an explicit formula for the convolution kernel of the fractional Laplacian. From Theorem \ref{thm-ST2} we have that obtaining this kernel is equivalent to compute the Laplace transform of the function \\
$ \ \prod_{\alpha\in \Sigma_0^+}(1+\langle \alpha, H\rangle)(1+\langle \alpha, H\rangle+t)^{\frac{m_\alpha+m_{2\alpha}}{2}-1}
\frac{e^{-\langle \varrho, H\rangle-\frac{|H|^2}{4t}}}{t^{\frac n2+1+\gamma}}$ at $|\varrho|^2$ .

The Laplace transform of $g(t)=\frac{e^{-\langle \varrho, H\rangle-\frac{|H|^2}{4t}}}{t^{\frac n2+1+\gamma}}$ is given by the integral
$$\calL(g)(s)=\int_0^\infty g(t)e^{-st}dt.$$
Since $$t^2g'(t)=\lp \frac{|H|^2}{4}- (\frac n2+1+\gamma)t\rp g(t),$$
and $g(0)=0$, applying Laplace transform we have
$$\partial^2_s(s \calL(g))=\lp \frac{|H|^2}{4}+ (\frac n2+1+\gamma)\partial_s\rp \calL(g),$$
which is a Bessel type equation of solution  $$\calL(g)=C_{n,\gamma}s^{\frac{n+2\gamma}{4}} K_{\frac{n+2\gamma}{2}}\lp |H|s^{\frac{1}{2}}\rp, $$
with $K_\nu$ given by Lemma \ref{lemma-Bessel}.

On the other hand, a simple computation shows that the Laplace transform of  $(1+\langle \alpha, H\rangle+t)^{\frac{m_\alpha+m_{2\alpha}}{2}-1}$ is given by
$$\frac{e^{s+1+\langle \alpha, H\rangle}}{s^{\frac{m_\alpha+m_{2\alpha}}{2}-1}} \Gamma\lp\frac{m_\alpha+m_{2\alpha}}{2}\rp.$$
Since
 $$\calL(gh)(s)=\frac{1}{2\pi i}\lim_{T\to \infty}\int_{c-iT}^{c+iT}\calL(g)(\sigma)\calL(h)(\sigma-s)d\sigma,$$ where the integration is done along the vertical line $Re(z)=c$. We denote this convolution type of operation as  $\calL(g)\tilde{\star} \calL(h)$.

 Then, by denoting as $q_i=\frac{m_{\alpha_i}+m_{2\alpha_i}}{2}$ and $l$ the total number of roots in $\Sigma_0^+$,
 we can compute the desired Kernel as
$$\mathcal K_\gamma(e^H)=\left. \frac{e^{s+1+\langle \alpha_1, H\rangle}}{s^{q_1-1}} \Gamma\lp q_1\rp \tilde{\star}\ldots  \tilde{\star}\frac{e^{s+1+\langle \alpha_{l}, H\rangle}}{s^{q_{l}-1}} \Gamma\lp q_{l}\rp\tilde{\star} C_{n,\gamma}s^{\frac{n+2\gamma}{4}} K_{\frac{n+2\gamma}{2}}\lp |H|s^{\frac{1}{2}}\rp\right|_{s= |\varrho|^2}$$


\subsubsection{Geometrically finite hyperbolic manifolds}\label{Geometrically finite hyperbolic manifolds}

The following introduction is standard and we refer the reader to \cite{Bowditch,Ratcliffe}. Let $\Gamma$ be a discrete group of isometries of $\mathbb H^n$, that without loss of generality can be taken to be torsion-free. Then the quotient $M=\mathbb H^n / \Gamma$ is a smooth manifold which inherits a complete hyperbolic structure. If $x\in\mathbb H^n$, the set of accumulation points of the orbit $x\Gamma$ in $\overline {\mathbb H^n}$  is a closed subset $\Lambda(\Gamma)\subset \ms^{n-1}$ called the limit set of $\Gamma$. Its complement $\Omega(\Gamma)=\ms^{n-1}\backslash \Lambda(\Gamma)$ is called the domain of discontinuity and $\Gamma$ acts properly discontinuously in $\Omega(\Gamma)$.

We assume, in addition, that $\Gamma$ is \emph{geometrically finite}, i.e., it admits a fundamental domain with finitely many sides, and we consider those groups for which $M$ has infinite volume. In the case that no parabolic subgroup involves irrational rotations, which is the setting of \cite{Perry:kleinian-groups}  for the study of the resolvent of the Laplacian $\Delta_M$ (see also \cite{Guillarmou-Mazzeo} for more general admissible groups), it is easy to give a geometrical description. In fact, there exists a compact $K$ of $M$ such that $M\backslash K$ is covered by a finite number of charts isometric to either a regular neighborhood $(M_0,g_0)$, where
$$M_0=\{(x_1,x_2)\in(0,\infty)\times \mathbb R^{n-1} \,:\, x_1^2+|x_2|^2<1\},\quad g_0=(x_1)^{-2}(dx_1^2+dx_2^2),$$
or a rank-$r$ \emph{cusp} neighborhood $(M_r,g_r)$, $1\leq r\leq n-1$, where
$$M_r=\{(x_1,x_2,x_3)\in(0,\infty)\times \mathbb R^{n-1-r}\times T^r \,:\, x_1^2+|x_2|^2>1\},\,\, g_r=(x_1)^{-2}(dx_1^2+dx_2^2+dx_3^2),$$
for $r<n-1$ and
$$M_{n-1}=\{(x_1,x_3)\in(0,\infty)\times T^n \,:\, x_1>1\},\quad g_{n-1}=(x_1)^{-2}(dx_1^2+dx_3^2).$$
Here $(T^k,dx_3^2)$ is a $k$-dimensional flat torus.


When such $\Gamma$ has no parabolic elements, then both $\Gamma$ and the quotient $\mathbb H^n/\Gamma$ are called \emph{convex co-compact}, and the quotient manifold $M$ has no cusp then.

We define $\delta(\Gamma)$, the exponent of convergence of the Poincar\'e series, by
$$\delta(\Gamma)=\inf\left\{ s>0 \,:\, \sum_{h\in\Gamma}d_{\mathbb H}(x,hx')^{-s}<\infty\right\},$$
where $x,x'\in\mathbb H^n$. Note that it depends upon the group $\Gamma$ but not upon the choice of $x,x'$. It is known that $0\leq \delta(\Gamma)\leq n-1$. We also define $\mu_\alpha$, for $\alpha>0$, by
$$\mu_\alpha(x)=\left\{\sum_{h\in\Gamma} e^{-\alpha d_{\mathbb H^n}(x,hx)}\right\}^{1/2},$$
noting that the sum is invariant under the action of $\Gamma$ on $x$, so that $\mu_\alpha$ can also be regarded as a function on $\mathbb H^n/\Gamma$. Although the series converges for all $\alpha>\delta(\Gamma)$, we shall often assume that $\alpha>n-1$, because $\mu_\alpha$ is both smaller and easier to estimate for larger $\alpha$. We note that the distance function $\tilde \rho$ on $\mathbb H^n/\Gamma$ is given by
$$\tilde \rho(x,x')=\min_{h\in \Gamma} d_{\mathbb H^n}(hx,x').$$

We recall the following bounds for the heat kernel on $\mathbb H^n/\Gamma$ from \cite{Davies-Mandouvalos:kleinian}. For $t\in(0,\infty)$ and any $\epsilon>0$:
\begin{enumerate}
\item[\emph{i.}] If $0\leq \delta(\Gamma)<\frac{n-1}{2}$, then
$$0\leq \tilde p_t(x,x')\leq c_\epsilon t^{-\frac{n}{2}}e^{-\lp (n-1)^2/4-2\epsilon\rp t} e^{-\frac{\tilde\rho^2}{4(1+\epsilon)t}}\mu_\alpha(x)\mu_\alpha(x').$$
\item{\emph{ii.}} If $\frac{n-1}{2}\leq \delta(\Gamma)\leq n-1$ and $\alpha<\delta(\Gamma)$,
$$p_t(x,x')\leq c_\epsilon t^{-\frac{n}{2}}e^{-\lp [\delta(\Gamma)(n-1-\delta(\Gamma))]-2\epsilon\rp t} e^{-\frac{\tilde \rho^2}{4(1+\epsilon)t}}\mu_\alpha(x)\mu_\alpha(x').$$
\end{enumerate}
Moreover, if $\alpha>n-1$, then $\mu_\alpha(x)\to 1$ as $x$ approaches to an end, and $\mu_\alpha(x)\sim e^{\frac{1}{2}\rho(x,z)r}$ as $x$ approaches a cusp of rank $r$, where $z$ is any point in $\mathbb H^n$. In particular, $\mu_\alpha(x)$ is bounded if $\alpha>n-1$ and the manifold has no cusps.

In order to obtain $L^\infty$ estimates for the heat kernel near a cusp of rank $r$, we make the following observation:
$$e^{-\frac{\rho^2}{4(1+\epsilon)t}}e^{\frac{r}{2}\rho}=e^{-\lp \frac{\rho}{2\sqrt{(1+\epsilon)t}}-\frac{r\sqrt{(1+\epsilon)t}}{2}\rp^2}e^{\frac{r^2(1+\epsilon)}{4}t}.$$
We get good $L^\infty$ bounds for $p_t(x,\cdot)$ when:
\begin{enumerate}
\item[\emph{i.}]   $0\leq \delta(\Gamma)<\frac{n-1}{2}$,   $r<n-1$ and there is no maximal cusp. Or,
\item[\emph{ii.}]  $\delta(\Gamma)=\frac{n-1}{2}+\frac{\beta}{2}$ for some $\beta\in[0,n-1)$ and
$r<(n-1)^2-\beta^2.$
\end{enumerate}
Alternatively, we can obtain good bounds for the heat kernel if the manifold has no cusp.


\subsubsection{Rotationally symmetric manifolds}\label{Rotationally symmetric manifolds}

We consider a noncompact rotationally symmetric manifold with a pole at the origin, i.e. a manifold $M^n$ given by the metric
\begin{equation}
\label{metric}g_M=dr^2+\phi^2(r)\,d\omega^2,
\end{equation}
where $d\omega^2$ is the metric on $\mathbb S^{n-1}$, and $\phi$ is a $\mathcal C^\infty$ nonnegative function on $[0,\infty)$, strictly positive on $(0,\infty)$, such that $\phi^{(\text{even})}(0)=0$ and $\phi'(0)=1$. These conditions on $\phi$ ensure us that the manifold is
smooth (see section�$\,\S$1.3.4. of \cite{Petersen}). The volume element is $\phi^{n-1}(r)\,dr\,d\omega$, and the Laplace-Beltrami operator on $M$ is
$$\Delta_M = \partial_{rr}+(n-1)\frac{\phi'(r)}{\phi(r)}\partial_r+\frac{1}{\phi^2(r)}\Delta_{\mathbb S^{n-1}}.$$

For such manifold, the curvature of $M$ can be computed explicitly in terms of $\phi$ (see $\S$3.2.3 of \cite{Petersen}).
Indeed, there exists an orthonormal frame $(F_j)_{j=1}^n$ on $(M, g)$, where $F_n$ corresponds to the radial coordinate, and $F_1,\ldots,F_{n-1}$ to the spherical coordinates, for which $F_i\wedge F_j$ diagonalize the curvature operator $\mathcal R$ :
\begin{equation*}
\begin{split}
&\mathcal R(F_i\wedge F_n)=-\frac{\phi''}{\phi}F_i\wedge F_n,\quad i<n, \\
&\mathcal R(F_i\wedge F_j)=-\frac{(\phi')^2-1}{\phi^2} F_i\wedge F_j,\quad i,j<n,
\end{split}
\end{equation*}
which gives the sectional curvature
\begin{equation*}
\begin{split}
K_{ii}&=0,  \quad i=1,\ldots,n,\\
K_{ni}&=K_{in}=-\frac{\phi''}{\phi},\quad i=1,\ldots,n-1,\\
K_{ij}&=-\frac{(\phi')^2-1}{\phi^2}, \quad i,j=1,\ldots,n-1, i\neq j.
\end{split}
\end{equation*}
The Ricci curvature is then calculated as
\begin{equation*}
\begin{split}
Ric(F_i)&=-\lp(n-2)\frac{(\phi')^2-1}{\phi^2}+\frac{\phi''}{\phi}\rp F_i,\quad i<n,\\
Ric(F_n)&=-(n-1)\frac{\phi''}{\phi}F_n,
\end{split}
\end{equation*}
We denote by $S(r)$ and $B(r)$, $r>0$, the geodesic sphere and ball with center 0 and radius $r$, respectively. The volume of $B(r)$ and the area of $S(r)$ are calculated from
$$|B(r)|=\omega_n \int_0^r \phi(s)^{n-1}\,ds,\quad |S(r)|=\omega_n \phi(r)^{n-1},$$
where $\omega_n$ is the area of the standard unit sphere $\ms^{n-1}$.


In the particular case of Euclidean space, $\phi(r)=r$, while for hyperbolic space, $\phi(r)=\sinh r$. Typical examples are $\phi(r)=r+\beta r^\alpha$ for some constants $\alpha\geq 2$ and $\beta>0$, and
$$\phi(r)=\sum_{k=0}^n\frac{r^{(2k+1)}}{k!}$$ for some $n\in\mathbb R^n$, which serve as an interpolation space between Euclidean space and hyperbolic space. Their sectional curvature is non-positive and Ricci curvature is bounded below by a negative constant.

Rotationally symmetric manifolds are important in the sense that they serve as models for more general problems (see, for instance, \cite{Greene-Wu} for the study of manifolds having a pole, \cite{Hsu} for Brownian motion and probability aspects). However, we have not found in the literature a clean if and only if condition for having suitable heat kernel upper bounds (see \cite{Grigoryan-SaloffCoste1} for examples in this direction).

We define the ratios
$$f_1(r)=\frac{\phi''(r)}{\phi(r)},\quad f_2(r)=(n-2)\frac{(\phi')^2-1}{\phi^2}+\frac{\phi''}{\phi}.$$
Recalling Proposition \ref{prop:summary}, we are in a good situation if, for instance: both $f_1,f_2$ are uniformly upper-bounded and the volume control \eqref{condition-balls} holds.



One may compare the volume of any ball to the volume of a ball centered at the origin using \eqref{comparison}.
First, if $1\geq 2d_M(x,0)$, then
\begin{equation*}|B(0,1)|\leq |B(0,1-d_M(x,0))|\frac{|B_\beta(1)|}{\,|B_\beta(1-d_M(x,0))|\,}\leq C_n |B(x,1)|.\end{equation*}
On the other hand, if $1<2d_M(x,0)$, then
\begin{equation*}|B(0,1)|\leq |B(0,1 /4)| \frac{|B_\beta(1)|}{|\,B_\beta(1/4)\,|}\leq C_n |B(x,d_M(x,0)+1/4)|,\end{equation*}
and if we use comparison again,
\begin{equation*}|B(x,d_M(x,0)+1/4)|\leq |B(x,1/4)|\frac{|B_\beta(d_M(x,0)+1/4)|}{|B_\beta (1/4)|}\leq C_n |B (x,1)| \lp 1+d_M(x,0)\rp^n\end{equation*}
Therefore
\begin{equation}\label{compbis}
\frac 1{|B(x,1)|}\leq \frac 1{|B(0,1)|}\,C_n\, \lp 1+d_M(x,0)\rp^{n},
\end{equation}
so \eqref{condition-balls} is fulfilled. Summarizing we have obtained:

\begin{prop}
Let $M$ be a (noncompact) rotationally symmetric manifold with metric \eqref{metric}, with $\phi$ as given at the beginning of the section. Then $M$ satisfies condition \eqref{hypo3} in Theorem \ref{thm-ST2} if $f_1(r),f_2(r)$ are uniformly bounded from above for $r\in[0,\infty)$.
\end{prop}


\begin{rem}
Under conditions on the sectional curvature,
estimates have been obtained also on the Schrodinger operator
$e^{it\Delta_M}$ in \cite{Banica-Duyckaerts}.
\end{rem}

\begin{rem}\label{rotsym}In the case of rotationally symmetric manifolds there is an equivalent way to write the extension problem: we define the weight
$$w(r)=\left(\frac{r}{\phi(r)}\right)^\frac{n-1}{2},$$
 that provides an isometry between $L^2(M)$ and $L^2(\mathbb R^n)$. Then one has the conjugation formula
$$\Delta_M h= w\,L\,(w^{-1}h),$$
with
$$L=\partial_{rr}-\frac{n-1}{r}\partial_r+\frac{1}{\phi^2}\Delta_{\mathbb S^{n-1}}-V,$$
and $V$ the radial function  (see for instance \cite{Banica-Duyckaerts}),
$$V(r)=\frac{n-1}{2}\frac{\phi''}{\phi}-\frac{(n-1)(n-3)}{4}
\left(\left(\frac{\phi'}{\phi}\right)^2-\frac {1}{r^2}\right).$$
On one hand we obtain, if the operators are positive (which is the case if for instance $V$ is nonnegative),
$$(-\Delta_M)^\gamma h= w\,\left(L\right)^\gamma\,(w^{-1}h).$$
On the other hand, by the change of function $v(r,\omega,y)=w(r)\, u(r,\omega,y)$,
the system \eqref{extension1} is transformed into
\begin{equation}\label{extpotential}
\left\{\begin{split}
&\partial_{yy} u + \frac{a}{y}\, \partial_y u +L\,u\,(x,y)=0 \hbox{ for } (x,y)\in \mathbb R^n\times \rr_+,\\
&u(x,0)=w^{-1}(|x|)f(x) \qquad \hbox{ for } x\in \mathbb R^n.
\end{split}\right.
\end{equation}
We apply again the results of \cite{StingaTorrea}, this time for the operator $L$ acting on functions on $\mathbb R^n$, and obtain the existence of the solution of \eqref{extpotential} (and implicitly of \eqref{extension}), together with the limit
\begin{equation}\label{Neumannpotential}\left(L\right)^\gamma (w^{-1}f)=-d_{\gamma}\lim_{y\to 0}y^a \,\partial_y u,\end{equation}
In view of the definition of $u$ we get \eqref{Neumann}:
$$-d_{\gamma}\lim_{y\to 0}y^a \,\partial_y v= -w \,d_{\gamma}\lim_{y\to 0}y^a \,\partial_y u=w\left(L\right)^\gamma (w^{-1}f)=(-\Delta_M)^\gamma f.$$
Finally, taking Fourier transform on $\mathbb R^n$, \eqref{extpotential} gets transformed into
$$\partial_{yy} \hat u + \frac{a}{y}\, \partial_y \hat u -\lambda^2\,\hat u-\hat V\star\hat u = 0,$$
which, a priori, may not be explicitly solved, but it may give further information.
\end{rem}

Rotationally symmetric manifolds are toy models of warped products and one could continue this study further. We refer the reader to \cite{Grigoryan:weighted}.




\subsubsection{Manifolds with ends}\label{Manifolds with ends}

Let us consider first the case that $M$ is topologically of the form $X\times (0,\infty)$, where the manifold $X$ need not be compact; note that the extension to several cusps is straightforward. Assume that $X$ has dimension $N\geq 2$ and carries a metric $g_X$. Define the metric on $M$ as
$$g_M=\gamma(x,r)(g_X+dr^2),\quad x\in M,r\in(0,\infty).$$
In addition, we assume that $M$ is approximately hyperbolic in the sense that $\gamma$ is a positive $\mathcal C^\infty$ function which satisfies
$$c^{-1} r^{-2} \leq \gamma(x,r)\leq c r^{-2}, \quad |\partial_r\gamma|\leq c r^{-3} \quad \text{on }M $$
for some $c>0$. The Laplace-Beltrami operator is written as
$$\Delta_M=\gamma^{-\frac{N+1}{2}}\divergence_x (\gamma^{\frac{N-1}{2}}\nabla_x)
+\gamma^{-\frac{N+1}{2}}\partial_r (\gamma^{\frac{N-1}{2}}\partial_r).$$
It is shown in \cite{Davies-Mandouvalos:cusps} (see also \cite{Muller:cusps}) that if the heat kernel of $X$ satisfies the bound
$$0\leq p_t(x,x')\leq c t^{-N/2},\quad \mbox{for all }0< t\leq 1,\; x,x'\in X,$$
which is the case when $X$ is compact, then the heat kernel of $M$ satisfies
$$0\leq p_t(m,m')\leq c_\delta (1+r)^{N/2} (1+r')^{N/2}t^{-\frac{N+1}{2}} e^{(2\delta-\lambda_1)t}e^{-\frac{d_M(m,m')^2}{4(1+\delta)t}},$$
for all $0<t<\infty$, $m,m'\in M$, where $\lambda_1$ is the bottom of the spectrum of $-\Delta_M$. Since one can get as for \eqref{compbis} the estimate
\begin{equation*}
\frac 1{|B(x,\sqrt{t})|}\leq \frac 1{|B(0,\sqrt{t})|}\,C_n\, \lp 1+\frac{d_M(x,0)}{\sqrt{t}}\rp^{n},
\end{equation*}
plugging it into \eqref{upper-heat2} hat we can satisfy condition \eqref{hypo3} in Theorem \ref{thm-ST2}.\\


More generally, one may consider weighted complete manifolds of the form
$$M=M_1\#\ldots\#M_k,$$
that is, manifolds that are the connected sum of a finite number of manifolds $M_i$, $1\leq i\leq k$. More precisely, this means that $M$ is the disjoint union $M=K\cup E_1\cup\ldots\cup E_k$, where $K$ is a compact with smooth boundary (we refer to it as the central part) and each $E_i$ is isometric to the complement of a compact set $K_i$ with smooth boundary in $M_i$. If $M$ is weighted then we assume that the $M_i$'s are weighted. The weight on $M$ and the weight on $M_i$ coincide on $E_i$ (with the obvious identification). The goal of  \cite{Grigoryan-SaloffCoste} is to study heat kernel bounds for $M$ from the bounds on each $M_i$ through a gluing technique. In order to keep the presentation simple, we will not state their theorem in full generality, but explain the model cases that inspire it.

For an integer $m\in[1,N]$ we define the manifold $\mathcal R^m$ by
$$\mathcal R^1=\mathbb R_+\times \mathbb S^{N-1},\quad \mathcal R^m=\mathbb R^m\times \mathbb S^{N-m},\,m\geq 2.$$
The manifold $\mathcal R^m$ has topological dimension $N$ but its ``dimension at infinity" is $m$ in the sense that $V(x,r)\sim r^m$ for $r\geq 1$. This enables to consider finite connected sums
$M=\mathcal R^{N_1}\# \ldots \#\mathcal R^{N_k}$, for fixed $M$ and $k$ integers $N_1,N_2,\ldots,N_k\in[1,N]$.

We assume that all ends of $M$ are non-parabolic, i.e., each $N_i> 2$, and set
$$n:=\min_{1\leq i\leq k} N_i (>2).$$
Let $K$ be the central part of $M$ and $E_1,\ldots,E_k$ be the ends of $M$ so that $E_i$ is isometric to the complement of a compact set in $\mathcal R^{N_i}$. With some abuse of notation, we write $E_i=\mathcal R^{N_i}\backslash K$. For any point in $x\in M$, set
$$|x|:=\sup_{x'\in K} d(x,x').$$
Observe that since $K$ has non-empty interior, $x$ is separated from 0 on $M$ and $|x|\sim 1+d(x,K)$.

For instance, let $k=2$ (i.e., $M$ has two ends). Set $M=\mathcal R^n\#\mathcal R^m$, $2<n\leq m$. Assume that
$x\in\mathcal R^n\backslash K$, $x'\in\mathcal R^m\backslash K$, $t\geq t_0$.
Then we have the heat kernel bound
$$0\leq p_t(x,x')\leq C \lp \frac{1}{t^{m/2}|x|^{n-2}}+\frac{1}{t^{n/2}|x'|^{m-2}}\rp e^{-c\frac{d^2}{t}},$$
which is enough for our purposes.

\subsection{Other frameworks}

One may also consider the construction of the fractional Laplacian on metric graphs. Note that heat kernel estimates are also valid for the Laplacian on graphs, see for instance the recent studies \cite{Frank-Kovarik}, \cite{Keller-Lenz-Wojciechowski} and references therein as \cite{Urakawa:infinite-graphs}.

Finally, it would be interesting to construct real and complex powers of the complex Laplacian on K\"ahler manifolds.





\begin{thebibliography}{AAA}

\bibitem{Abramowitz-Stegun}
M.~Abramowitz and I.~A. Stegun.
\newblock {\em Handbook of mathematical functions with formulas, graphs, and
  mathematical tables}, volume~55 of {\em National Bureau of Standards Applied
  Mathematics Series}.
\newblock For sale by the Superintendent of Documents, U.S. Government Printing
  Office, Washington, D.C., 1964.

\bibitem{Anker:multipliers}
J.-P. Anker.
\newblock {${\bf L}_p$} {F}ourier multipliers on {R}iemannian symmetric spaces
  of the noncompact type.
\newblock {\em Ann. of Math. (2)}, 132(3):597--628, 1990.

\bibitem{Anker-Damek-Yacoub}
J.-P. Anker, E. Damek, and C. Yacoub.
\newblock Spherical analysis on harmonic {$AN$} groups.
\newblock {\em Ann. Scuola Norm. Sup. Pisa Cl. Sci. (4)}, 23(4):643--679, 1997.

\bibitem{Anker-Ji}
J.-P. Anker and L.~Ji.
\newblock Heat kernel and {G}reen function estimates on noncompact symmetric
  spaces.
\newblock {\em Geom. Funct. Anal.}, 9(6):1035--1091, 1999.

\bibitem{Anker-Pierfelice-Vallarino:Damek-Ricci}
J.-P. Anker, V. Pierfelice, and M. Vallarino.
\newblock Schr\"odinger equations on {D}amek-{R}icci spaces.
\newblock {\em Comm. Partial Differential Equations}, 36(6):976--997, 2011.

\bibitem{Anker-Ostellari}
J.-P. Anker and P.~Ostellari.
\newblock The heat kernel on noncompact symmetric spaces. Lie groups and symmetric spaces
\newblock {\em Amer. Math. Soc. Transl. Ser. 2}, 210:
    27--46, 2003.



\bibitem{Banica} V. Banica,
\newblock  The nonlinear Schr\"odinger equation on hyperbolic space,
\newblock {\em Comm. Partial Differential Equations}, 32:1643-1677, 2007.

\bibitem{Banica-Duyckaerts}
V. Banica and T. Duyckaerts.
\newblock Weighted Strichartz estimates for radial Schr\"odinger equation on noncompact manifolds.
\newblock {\em Dyn. Partial Differ. Equ.} 4(4):335-359, 2007.

\bibitem{Bowditch}
B.~H. Bowditch.
\newblock Geometrical finiteness for hyperbolic groups.
\newblock {\em J. Funct. Anal.}, 113(2):245--317, 1993.





\bibitem{BerensteinTarabusi}
 C. Berenstein, E. Casadio Tarabusi.
 \newblock Inversion formulas for the k-dimensional Radon
transform in real hyperbolic spaces.
\newblock {\em Duke Math. J.}(3):613--631, 1991.

\bibitem{CS}
L. A. Caffarelli and L. E. Silvestre.
\newblock An extension problem related to the fractional laplacian.
{\em Communications in Partial Differential Equations}, 32(8): 1245--1260, 2007.

\bibitem{Cabre-Sire:I}
X.~Cabr{\'e} and Y.~Sire.
\newblock Non-linear equations for fractional {L}aplacians {I}: regularity,
  maximum principles and {H}amiltoniam estimates.
\newblock  Preprint, arXiv:1012.0867.


\bibitem{Castorina-Fabbri-Mancini-Sandeep}
D.~Castorina, I.~Fabbri, G.~Mancini, and K.~Sandeep.
\newblock Hardy-{S}obolev extremals, hyperbolic symmetry and scalar curvature
  equations.
\newblock {\em J. Differential Equations}, 246(3):1187--1206, 2009.

\bibitem{Cheeger-Yau}
J. Cheeger and S.~T. Yau.
\newblock A lower bound for the heat kernel.
\newblock {\em Comm. Pure Appl. Math.}, 34(4):465--480, 1981.

\bibitem{Cheeger-Gromov-Taylor}
J. Cheeger, M. Gromov, and M. Taylor.
\newblock Finite propagation speed, kernel estimates for functions of the
  {L}aplace operator, and the geometry of complete {R}iemannian manifolds.
\newblock {\em J. Differential Geom.}, 17(1):15--53, 1982.

\bibitem{Chang-Gonzalez}
A. Chang and M.d.M. Gonzalez.
\newblock {\em Fracional Laplacian in conformal geometry.}
\newblock {\em Advances in Mathematics}, 226(2):1410--1432, 2011.

\bibitem{Cheng-Li-Yau}
S.~Y. Cheng, P. Li, and S.~T. Yau.
\newblock On the upper estimate of the heat kernel of a complete {R}iemannian
  manifold.
\newblock {\em Amer. J. Math.}, 103(5):1021--1063, 1981.

\bibitem{Cohen-Lifshits}
S.~Cohen and M.A. Lifshits.
\newblock Stationary {G}aussian random fieldson hyperbolic spaces and on
  {E}uclidean spheres.
\newblock {\em ESAIM: Probability and Statistics}, 16, 2012.

\bibitem{Cotsiolis-Tavoularis:best-constants}
A.~Cotsiolis and N.~K. Tavoularis.
\newblock Best constants for {S}obolev inequalities for higher order fractional
  derivatives.
\newblock {\em J. Math. Anal. Appl.}, 295(1):225--236, 2004.

\bibitem{Damek-Ricci}
E.~Damek and  F.~Ricci.
\newblock Harmonic analysis on solvable extensions of $H$-type groups.
\newblock{\em J. Geom. Anal.}, 2(3):213--248, 1992.

\bibitem{Davies:book}
E.~B. Davies.
\newblock {\em Heat kernels and spectral theory}, volume~92 of {\em Cambridge
  Tracts in Mathematics}.
\newblock Cambridge University Press, Cambridge, 1989.

\bibitem{Davies:upper-bounds}
E.~B. Davies.
\newblock Gaussian upper bounds for the heat kernels of some second-order
  operators on {R}iemannian manifolds.
\newblock {\em J. Funct. Anal.}, 80(1):16--32, 1988.

\bibitem{Davies-Mandouvalos:kleinian}
E.~B. Davies and N.~Mandouvalos.
\newblock Heat kernel bounds on hyperbolic space and Kleinian groups.
\newblock{\em Proc. London Math. Soc.} 57:192-208, 1988.

\bibitem{Davies-Mandouvalos:cusps}
E.~B. Davies and N.~Mandouvalos.
\newblock Heat kernel bounds on manifolds with cusps.
\newblock {\em J. Funct. Anal.}, 75(2):311--322, 1987.

\bibitem{Debiard-Gaveau-Mazet}
A.~Debiard, B.~Gaveau, and E.~Mazet.
\newblock Th\'eor\`emes de comparaison en g\'eom\'etrie riemannienne.
\newblock {\em Publ. Res. Inst. Math. Sci.}, 12(2):391--425, 1976/77.

\bibitem{Hitchhikers}
E. Di Nezza, G. Palatucci and E. Valdinoci.
\newblock Hitchhiker�s guide to the fractional Sobolev spaces.
\newblock {\em Bull. Sci. math.}, 136(5):521--573, 2012.

\bibitem{FKS}
E.~B. Fabes, C.~E. Kenig, and R.~P. Serapioni.
\newblock The local regularity of solutions of degenerate elliptic equations.
\newblock {\em Comm. Partial Differential Equations}, 7(1):77--116, 1982.

\bibitem{FJK}
E.~B. Fabes, C.~E. Kenig, and D.~Jerison.
\newblock Boundary behavior of solutions to degenerate elliptic equations.
\newblock In {\em Conference on harmonic analysis in honor of {A}ntoni
  {Z}ygmund, {V}ol. {I}, {II} ({C}hicago, {I}ll., 1981)}, Wadsworth Math. Ser.,
  pages 577--589. Wadsworth, Belmont, CA, 1983.

\bibitem{Folland}
G.~B~.Folland.
\newblock Subelliptic estimates and function spaces on nilpotent Lie groups.
\newblock{\em Ark. Mat.}   13(2):161--207, 1975.

\bibitem{Frank-Kovarik}
R.~L.~Frank and H.~Kovarik.
\newblock Heat kernels of metric trees and applications.
\newblock{\em SIAM J. Math. Anal.} 45(3):1027--1046, 2013.

\bibitem{GGG} I.M.  Gelfand, S.G. Gindikin, I.M. Graev,  \newblock{\em Selected topics in integral geometry. }  \newblock Translations of Mathematical Monographs, 220. American Mathematical Society, Providence, RI, 2003.

\bibitem{Gonzalez-Mazzeo-Sire}
M.~d.~M. Gonz\'alez, R.~Mazzeo and Y.~Sire.
\newblock Singular solutions of fractional order conformal laplacians.
\newblock {\em Journal of Geometric Analysis}, 22(3):845-863, 2012.

\bibitem{Gonzalez-Qing}
M.~d.~M. Gonz\'alez and J. Qing
\newblock Fractional conformal Laplacians and fractional Yamabe problems.
\newblock To appear in {\em Analysis and PDE}.

\bibitem{Gonzalez-Saez-Sire}
M.~d.~M. Gonz\'alez, M.~S\'aez and Y.~Sire.
\newblock Layer solutions for the fractional Laplacian on hyperbolic space: existence, uniqueness and qualitative properties.
\newblock To appear in {\em Annali di Matematica Pura ed Applicata}.

\bibitem{Graham-Zworski:scattering-matrix}
C.~R. Graham and M.~Zworski.
\newblock  Scattering matrix in conformal geometry.
\newblock  {\em Invent. Math.}, 152(1):89--118, 2003.

\bibitem{Greene-Wu}
R.~E. Greene and H.~Wu.
\newblock {\em Function theory on manifolds which possess a pole}, volume 699
  of {\em Lecture Notes in Mathematics}.
\newblock Springer, Berlin, 1979.

\bibitem{Grigoryan:weighted}
A. Grigor'yan.
\newblock Heat kernels on weighted manifolds and applications.
\newblock In {\em The ubiquitous heat kernel}, volume 398 of {\em Contemp.
  Math.}, pages 93--191. Amer. Math. Soc., Providence, RI, 2006.

\bibitem{Grigoryan1999}
A. Grigor'yan.
\newblock Analytic and geometric background of recurrence and non-explosion of
  the {B}rownian motion on {R}iemannian manifolds.
\newblock {\em Bull. Amer. Math. Soc. (N.S.)}, 36(2):135--249, 1999.


\bibitem{Grigoryan1993}
A.~Grigor'yan.
\newblock Gaussian upper bounds for the heat kernel and for its derivatives on
  a {R}iemannian manifold.
\newblock In {\em Classical and modern potential theory and applications
  ({C}hateau de {B}onas, 1993)}, volume 430 of {\em NATO Adv. Sci. Inst. Ser. C
  Math. Phys. Sci.}, pages 237--252.

\bibitem{Grigoryan1995}
A. Grigor'yan.
\newblock Upper bounds of derivatives of the heat kernel on an arbitrary complete manifold,
\newblock {\em J. Funct. Anal.} 127(2): 363--389. 1995.

\bibitem{Grigoryan1998:survey}
A. Grigor'yan.
\newblock Estimates of heat kernels on {R}iemannian manifolds.
\newblock In {\em Spectral theory and geometry ({E}dinburgh, 1998)}, volume 273
  of {\em London Math. Soc. Lecture Note Ser.}, pages 140--225. Cambridge Univ.
  Press, Cambridge, 1999.

\bibitem{Grigoryan-SaloffCoste}
A. Grigor'yan and L. Saloff-Coste.
\newblock Heat kernel on manifolds with ends.
\newblock {\em Ann. Inst. Fourier (Grenoble)}, 59(5):1917--1997, 2009.

\bibitem{Grigoryan-SaloffCoste1}
A. Grigor'yan and L. Saloff-Coste.
\newblock Stability results for Harnack inequalities.
\newblock {\em Ann. Inst. Fourier (Grenoble)}, 55:825--890, 2005.

\bibitem{Grunau-OuldAhmedou-Reichel}
H.-C. Grunau, M. Ould~Ahmedou, and W. Reichel.
\newblock The {P}aneitz equation in hyperbolic space.
\newblock {\em Ann. Inst. H. Poincar\'e Anal. Non Lin\'eaire}, 25(5):847--864,
  2008.

\bibitem{Guillarmou-Mazzeo}
C. Guillarmou and R. Mazzeo.
\newblock Resolvent of the {L}aplacian on geometrically finite hyperbolic
  manifolds.
\newblock {\em Invent. Math.}, 187(1):99--144, 2012.


\bibitem{HRiesz}
 S. Helgason,
 \newblock The totally-geodesic Radon transform on constant curvature spaces.
\newblock  {\em Contemp. Math. 113}(1990), 141-149.

\bibitem{H}
S. Helgason. \newblock {\em Geometric analysis on symmetric spaces. }
 \newblock Second edition. Mathematical Surveys and Monographs, 39. American Mathematical Society, Providence, RI, 2008. xviii+637 pp.

\bibitem{Helgason:libro3}
S. Helgason.
\newblock {\em Groups and geometric analysis}, volume~83 of {\em Mathematical
  Surveys and Monographs}.
\newblock American Mathematical Society, Providence, RI, 2000.
\newblock Integral geometry, invariant differential operators, and spherical
  functions, Corrected reprint of the 1984 original.

\bibitem{Helgason:libro2}
S. Helgason.
\newblock {\em Differential geometry, {L}ie groups, and symmetric spaces},
  volume~34 of {\em Graduate Studies in Mathematics}.
\newblock American Mathematical Society, Providence, RI, 2001.
\newblock Corrected reprint of the 1978 original.


\bibitem{Hsu}
E. Hsu.
\newblock {\em Stochastic analysis on manifolds}, volume~38 of {\em Graduate
  Studies in Mathematics}.
\newblock American Mathematical Society, Providence, RI, 2002.

\bibitem{Keller-Lenz-Wojciechowski}
M. Keller, D. Lenz and R. Wojciechowski.
\newblock {Volume growth, spectrum and schochastic completeness of infinite graphs.}
\newblock Preprint.
\newblock{\em ArXiv:1105.0395v1}

\bibitem{Landkof}
N.~S. Landkof.
\newblock {\em Foundations of modern potential theory}.
\newblock Springer-Verlag, New York, 1972.
\newblock Translated from the Russian by A. P. Doohovskoy, Die Grundlehren der
  mathematischen Wissenschaften, Band 180.


\bibitem{Li-Yau}
P. Li and S.-T. Yau.
\newblock On the parabolic kernel of the {S}chr\"odinger operator.
\newblock {\em Acta Math.}, 156(3-4):153--201, 1986.

\bibitem{Liu:Sobolev-hyperbolic}
G. Liu.
\newblock Sharp higher-order {S}obolev inequalities in the hyperbolic space
  $\mathbb H^n$.
\newblock {\em Calc. Var. Partial Differential Equations}, 2012.

\bibitem{Mancini-Sandeep}
G. Mancini and K. Sandeep.
\newblock On a semilinear elliptic equation in {$\Bbb H^n$}.
\newblock {\em Ann. Sc. Norm. Super. Pisa Cl. Sci. (5)}, 7(4):635--671, 2008.

\bibitem{Mancini-Sandeep:Sobolev}
G. Mancini and K. Sandeep.
\newblock Extremals for {S}obolev and {M}oser inequalities in hyperbolic space.
\newblock {\em Milan J. Math.}, 79(1):273--283, 2011.

\bibitem{Mazzeo-Melrose:meromorphic-extension}
R. Mazzeo and R. Melrose.
\newblock Meromorphic extension of the resolvent on complete spaces with
  asymptotically constant negative curvature.
\newblock {\em J. Funct. Anal.}, 75(2):260--310, 1987.

\bibitem{Muller:cusps}
W. M{\"u}ller.
\newblock Spectral theory for {R}iemannian manifolds with cusps and a related
  trace formula.
\newblock {\em Math. Nachr.}, 111:197--288, 1983.



\bibitem{Murata:2011}
M. Murata.
\newblock Nonnegative solutions of the heat equation on rotationally symmetric
  {R}iemannian manifolds and semismall perturbations.
\newblock {\em Rev. Mat. Iberoam.}, 27(3):885--907, 2011.



\bibitem{Perry:kleinian-groups}
P. Perry.
\newblock Meromorphic continuation of the resolvent for {K}leinian groups.
\newblock In {\em Spectral problems in geometry and arithmetic ({I}owa {C}ity,
  {IA}, 1997)}, volume 237 of {\em Contemp. Math.}, pages 123--147. Amer. Math.
  Soc., Providence, RI, 1999.

\bibitem{Petersen}
P. Petersen.
\newblock {\em Riemannian geometry}, volume 171 of {\em Graduate Texts in
  Mathematics}.
\newblock Springer-Verlag, New York, 1998.

\bibitem{Pinchover}
Y. Pinchover.
\newblock Maximum and anti-maximum principles and eigenfunctions estimates via
  perturbation theory of positive solutions of elliptic equations.
\newblock {\em Math. Ann.}, 314(3):555--590, 1999.


\bibitem{Ratcliffe}
J. Ratcliffe.
\newblock {\em Foundations of hyperbolic manifolds}, volume 149 of {\em
  Graduate Texts in Mathematics}.
\newblock Springer-Verlag, New York, 1994.

\bibitem{Roncal-Stinga:torus}
L. Roncal and P. R. Stinga.
\newblock Fractional {L}aplacian on the torus.
\newblock Preprint, 2012.


\bibitem{Rudin:book}
W. Rudin.
\newblock {\em Functional analysis}.
\newblock McGraw-Hill Book Co., New York, 1973.
\newblock McGraw-Hill Series in Higher Mathematics.

\bibitem{Saloff-Coste:book}
L. Saloff-Coste.
\newblock{\em The heat kernel and its estimates}.
\newblock Advanced Studies in Mathematics, 2009.


\bibitem{Silvestre:regularity-obstacle}
L. Silvestre.
\newblock Regularity of the obstacle problem for a fractional power of the
  {L}aplace operator.
\newblock {\em Comm. Pure Appl. Math.}, 60(1):67--112, 2007.

\bibitem{StingaTorrea}
P.~R. Stinga and J.~L. Torrea.
\newblock Extension problem and Harnack's inequality for some fractional operators,
\newblock {\em Commun. Partial Differ. Equations} 35(10-12):2092--2122, 2010.


\bibitem{Tataru:Strichartz-hyperbolic}
D. Tataru.
\newblock Strichartz estimates in the hyperbolic space and global existence for
  the semilinear wave equation.
\newblock {\em Trans. Amer. Math. Soc.}, 353(2):795--807 (electronic), 2001.


\bibitem{Terras:book}
A. Terras.
\newblock {\em Harmonic analysis on symmetric spaces and applications. {I}}.
\newblock Springer-Verlag, New York, 1985.

\bibitem{Triebel:II}
H. Triebel.
\newblock {\em Theory of function spaces. {II}}, volume~84 of {\em Monographs
  in Mathematics}.
\newblock Birkh\"auser Verlag, Basel, 1992.

\bibitem{Urakawa:infinite-graphs}
H. Urakawa.
\newblock Heat kernel and {G}reen kernel comparison theorems for infinite
  graphs.
\newblock {\em J. Funct. Anal.}, 146(1):206--235, 1997.

\bibitem{Valdinoci:long-jump}
E. Valdinoci.
\newblock From the long jump random walk to the fractional {L}aplacian.
\newblock {\em Bol. Soc. Esp. Mat. Apl. S$\vec{\rm e}$MA}, (49):33--44, 2009.


\bibitem{Varopoulos}
N. Varopoulos,
\newblock Semi-groupes d'op\'erateurs sur les espaces $L^p$.
\newblock {\em C. R. Acad. Sci. Paris S\'er. I Math.}, (301):865-868, 1985.


\bibitem{Yosida:book}
K. Yosida.
\newblock {\em Functional analysis}.
\newblock Die Grundlehren der Mathematischen Wissenschaften, Band 123. Academic
  Press Inc., New York, 1965.

\end{thebibliography}
\end{document}